\documentclass[12pt,a4paper,twoside]{amsart}

\usepackage{amsmath,amssymb}
\usepackage{tikz,pgf}

\theoremstyle{plain}
\newtheorem{lemma}{Lemma}[section]
\newtheorem{proposition}{Proposition}[section]
\newtheorem{theorem}{Theorem}[section]
\newtheorem{corollary}{Corollary}[section]

\theoremstyle{definition}
\newtheorem{definition}{Definition}[section]

\newtheorem{remark}{Remark}[section]
\newtheorem{example}{Example}[section]
\newtheorem{assumption}{Assumption}[section]

\newcommand{\Z}{\mathbb{Z}}
\newcommand{\C}{\mathbb{C}}

\newcommand{\R}{\mathbb{R}}

\newcommand{\sgn}{\mathrm{sgn\,}}

\newcommand{\ceil}[1]{\ensuremath{\left\lceil #1 \right\rceil}}
\renewcommand{\vec}[1]{ \ensuremath{ {\bf#1}} }
\newcommand{\selberg}[1]{\ensuremath{\left\langle #1 \right\rangle}}

\title{Asymptotics for the number of Walks in a Weyl chamber of Type $B$}
 \author[Thomas Feierl]{Thomas Feierl$^\ddagger$}
\thanks{$^\ddagger$ Fakult\"at f\"ur Mathematik, Universit\"at Wien}
\address{Thomas Feierl \\ Fakult\"at f\"ur Mathematik \\ Universit\"at Wien \\ Nordbergstr. 15 \\ 1090 Wien \\ Austria}
\date{December 15, 2010}

\begin{document}
\maketitle
\begin{abstract}
We consider lattice walks in $\R^k$ confined to the region $0<x_1<x_2...<x_k$
with fixed (but arbitrary) starting and end points. The walks are required 
to be "reflectable", that is, we assume that the number of paths can be 
counted using the reflection principle.
The main results are asymptotic formulas for the total number of walks of
length $n$ with either a fixed or a free end point for a general class of walks as $n$ tends to infinity.
As applications, we find the asymptotics for the number of $k$-non-crossing
tangled diagrams on the set $\{1,2,...,n\}$ as $n$ tends to infinity,
and asymptotics for the number of $k$-vicious walkers subject to a wall restriction
in the random turns model as well as in the lock step model.
Asymptotics for all of these objects were either known only for certain special cases, or have only been partially determined or were completely unknown.
\end{abstract}

\section{Introduction}

Lattice paths are well-studied objects in combinatorics as well as in probability theory. A typical problem that is often encountered
is the determination of the number of lattice paths that stay within a certain fixed region. In many situations, this region can be
identified with a Weyl chamber corresponding to some reflection group. In this paper, the region is a Weyl chamber of type $B$, and, more precisely,
it is given by $0<x_1<\dots<x_k$. (Here, $x_j$ refers to the $j$-th coordinate in $\R^k$.)

Under certain assumptions on the set of allowed steps and on the underlying lattice, the total number of paths as described above can
be counted using the \emph{reflection principle} as formulated by Gessel and Zeilberger~\cite{MR1092920}. 
This reflection principle is a generalisation of a reflection argument, which is often attributed to Andr{\'e}~\cite{andre}, to the context
of general finite reflection groups (for details on reflection groups, see \cite{MR1066460}).

A necessary and sufficient condition on the set of steps for ensuring the applicability of the reflection principle as formulated by Gessel and Zeilberger~\cite{MR1092920} has been given by Grabiner and Magyar~\cite{MR1235279}. 
In their paper, Grabiner and Magyar also stated a precise list of steps that satisfy these conditions.

In a recent paper that attracted the author's interest, and that was also the main initial motivation for this work,
Chen et al.~\cite[Obervations~1 and 2]{ChQiReiZeil} gave lattice path descriptions for combinatorial objects called
\emph{$k$-non-crossing tangled diagrams}.
In their work, they determined the order of asymptotic growth of these objects, but they did not succeed in determining precise asymptotics.
Interestingly, the sets of steps appearing in this description do not satisfy Grabiner and Magyar's condition.
Nevertheless, a slightly generalised reflection principle turns out to be applicable because the steps can be interpreted as sequences of
certain \emph{atomic steps}, where these atomic steps satisfy Grabiner and Magyar's condition.
In this manuscript, we state a generalised reflection principle that applies to walks consisting of steps that are sequences of such atomic steps (see Lemma~\ref{lem:reflection_principle} below).

Our main results are asymptotic formulas for the total number of walks as the number of steps tends to infinity that stay within the region $0<x_1<\dots<x_k$, with either a fixed end point or a free end point (see Theorem~\ref{thm:asymptotic_u->v} and Theorem~\ref{thm:asymptotic_u->}, respectively).
The starting point of our walks may be chosen anywhere within the allowed region.
The proofs of the main results can be roughly summarised as follows. Using a generating function approach, we are able to express the number of walks that we are interested in as a certain coefficient in a specific Laurent polynomial.
We then express this coefficient as a Cauchy integral and extract asymptotics with the help of saddle point techniques.
Of course, there are some technical problems in between that we have to overcome.
The most significant comes from the fact that we have to determine asymptotics for a determinant.
The problem here is the large number of cancellations of asymptotically leading terms. It is surmounted by means
of a general technique that is presented in Section~\ref{sec:determinants_and_asymptotics}.
As a corollary to our main results, we obtain precise asymptotics for $k$-non-crossing tangled diagrams with and without isolated
points (for details, see Section~\ref{sec:applications}).
Moreover, we find asymptotics for the number of vicious walks with a wall restriction in the lock step model as well as asymptotics for 
the number of vicious walks with a wall restriction in the random turns model.
Special instances of our asymptotic formula for the total number of vicious walks in the lock step model have been established by Krattenthaler et al.~\cite{MR1801472,MR1964695} and Rubey~\cite{rubey}.
The growth order for the number of vicious walks in the lock step model with a free and point, and for the number of $k$-non-crossing tangled diagrams has been determined by Grabiner~\cite{grabiner} and Chen et al.~\cite{MR2426149}, respectively.
To the author's best knowledge, the asymptotics for the number of vicious walks in the random turns model seem to be new.

%
In some sense, one of the achievements of the present work is that it shows how to overcome a technical difficulty put to the fore in
\cite{MR2102573}. In order to explain this remark, we recall that
Tate and Zelditch~\cite{MR2102573} determined asymptotics of multiplicities of weights in tensor powers, which are related to 
reflectable lattice paths in a Weyl chamber (for details, we refer to \cite[Theorem 2]{MR1235279}).
For the so-called \emph{central limit region of irreducible multiplicities} (for definition, we directly refer to \cite{MR2102573})
they did not manage to determine the asymptotic behaviour of these multiplicities, and, therefore, had to resort
to a result of Biane~\cite[Th\'eor\`eme 2.2]{MR1218274}.
More precisely, although they were able to obtain the (indeed correct) dominant asymptotic term in a formal manner,
they were not able to actually prove its validity by establishing a sufficient bound on the error term.
For a detailed elaboration on this problem we refer to the paragraph after \cite[Theorem 8]{MR2102573}.
The techniques applied in \cite{MR2102573} are in fact quite similar to those applied in this manuscript (namely, the Weyl character formula/reflection principle and saddle point techniques).
However, it is the above mentioned technique presented in Section~\ref{sec:determinants_and_asymptotics} which forms the key to resolve the problem
by providing sufficiently small error bounds in situations as the one described by Tate and Zelditch~\cite{MR2102573}.

The paper is organised as follows. 
In the next section, we give the basic definitions and precise description of the lattice walk model underlying this work.
We also state and prove a slightly generalised reflection principle (see Lemma~\ref{lem:reflection_principle} below) that can be used to count the number of lattice walks in our model.
At the end of this section, we prove an exact integral formula for this number.
In Section~\ref{sec:auxiliary}, we we study more carefully step generating functions, determine their structure and draw some conclusions that are of importance in the proofs of our main results.
Section~\ref{sec:determinants_and_asymptotics} presents a factorisation technique for certain functions defined by determinants.
These results are crucial to our proofs since they enable us to determine precise asymptotics for these functions.
Our main results, namely asymptotics for total number of random walks with a fixed end point and with a free end point, are the content of Section~\ref{sec:fixedendpoint} and Section~\ref{sec:freeendpoint}, respectively.
The last section presents applications of our main results, namely Theorem~\ref{thm:asymptotic_u->v} and Theorem~\ref{thm:asymptotic_u->}.
Here we determine asymptotics for the number of vicious walks with a wall restriction in the lock step model as well as asymptotics for the number of vicious walks with a wall restriction in the random turns model.
Furthermore, we determine precise asymptotics for the number of $k$-non-crossing tangled diagrams on the set $\left\{ 1,2,\dots,n \right\}$
as $n$ tends to infinity.
This generalises results by Krattenthaler et al.~\cite{MR1801472,MR1964695} and Rubey~\cite{rubey}.
Additionally, we provide precise asymptotic formulas for counting problems for which only the asymptotic growth order has been established.
In particular, we give precise asymptotics for the total number of vicious walkers with wall restriction and free end point, as well as precise asymptotics for the number of $k$-non-crossing tangled diagrams with and without isolated points.
(The growth order for the former objects has been established by Grabiner~\cite{grabiner}, whereas the growth order for the latter objects has been determined by Chen et al.~\cite{ChQiReiZeil}.)

\section{Reflectable walks of type $B$}

The intention of this section is twofold.
First, we give a precise description of the lattice walk model underlying this work, and state some basic results.
Second, we derive an \emph{exact integral formula} (see Lemma~\ref{lem:exact_integral_u->v} below) for the generating function of lattice walks in this model with respect to a given weight.

Let us start with the presentation of the lattice path model.
We will have two kind of steps: atomic steps and composite steps.
\emph{Atomic steps} are elements of $\R^k$.
The set of all atomic steps in our model will always be denoted by $\mathcal A$.
\emph{Composite steps} are finite sequences of atomic steps.
The set of composite steps in our model will be always be denoted by $\mathcal S$.
Both sets, $\mathcal A$ and $\mathcal S$, are assumed to be finite sets.
By $\mathcal L$ we denote the $\Z$-lattice spanned by the atomic step set $\mathcal A$.

The walks in our model are walks on the lattice $\mathcal L$ consisting of steps from the composite step set $\mathcal S$ that are confined to the region
\[ \mathcal W^0=\left\{ (x_1,\dots,x_k)\in\R^k\ :\ 0<x_1<\dots<x_k \right\}. \]
For a given function $w:\mathcal S\to\R_+$, called the \emph{weight function}, we define the weight of a walk with step sequence $(\vec s_1,\dots,\vec s_n)\in\mathcal S^n$ by $\prod_{j=1}^nw(\vec s_j)$.
%
%

The generating function for all $n$-step paths from $\vec u\in\mathcal L$ to $\vec v\in\mathcal L$ with respect to the weight $w$ will be
denoted by $P_n(\vec u\to\vec v)$, that is,
\[
P_n(\vec u\to\vec v)=\sum_{\substack{\vec s_1,\dots,\vec s_n\in\mathcal S \\ \vec u+\vec s_1+\dots+\vec s_n=\vec v}}\prod_{j=1}^nw(\vec s_j),
\]
and the generating function of those paths of length $n$ from $\vec u$ to $\vec v$ with respect to the weight $w$ that stay within the region $\mathcal W^0$ will be denoted by $P_n^+(\vec u\to\vec v)$.

The ultimate goal of this work is the derivation of an asymptotic formula for $P_n^+(\vec u\to\vec v)$ as $n$ tends to infinity for certain step sets $\mathcal S$ and certain weight functions $w$.

In the theory of reflection groups (or Coxeter groups), $\mathcal W^0$ is called a \emph{Weyl chamber of type $B_k$}.
By $\mathcal W$, we denote the closure of $\mathcal W^0$, viz.
\[ \mathcal W=\left\{ (x_1,\dots,x_k)\in\R^k\ :\ 0\le x_1\le\dots\le x_k \right\}, \]
The boundary of $\mathcal W$ is contained in the union of the hyperplanes 
\begin{equation}\label{eq:cox_group_type_b:generators}
x_i-x_j=0\quad\textrm{for }1\le i<j\le k,\qquad\textrm{and }\quad x_1=0.
\end{equation}
The set of reflections in these hyperplanes is a generating set for the finite reflection group of type $B_k$ (see Humphreys~\cite{MR1066460}).
For an illustration, see Figure~\ref{fig:type_B_illustration}.
\begin{figure}
	\begin{center}
\begin{tikzpicture}
	\draw[lightgray,thin,step=1cm] (-3.5,-3.5) grid (3.5,3.5);
	\foreach \x in {-3,-2,...,3} {
		\foreach \y in {-3,-2,...,3} {
			\draw[gray,fill=gray] (\x,\y) circle (.05cm);
		}
	}
	\draw[->] (-3.5,0) -- (3.5,0) node[right] {$x_1$};	
	\draw[->] (0,-3.5) -- (0,3.5) node[above] {$x_2$};
	\draw[ultra thick,red] (-3.2,-3.2) -- (3.2,3.2);
	\draw[ultra thick,red] (0,-3.2) -- (0,3.2);
	\draw[ultra thick,red,dashed,opacity=0.8] (-3.2,0) -- (3.2,0);
	\draw[ultra thick,red,dashed,opacity=0.8] (3.2,-3.2) -- (-3.2,3.2);
	\draw[white,opacity=.15,fill=blue!50!white] (.05,.15) -- (3.1,3.2) -- (.05,3.2);
	\draw[blue] node[above left] at (1.5,2) {$\mathcal W$};
\end{tikzpicture}
	\end{center}
	\caption{Illustration of the group $B_2$. The shaded region indicates the Weyl chamber $0<x_1<x_2$. the border of $\mathcal W$ is contained in the two hyperplanes represented by the solid thick lines. The reflections in these hyperplanes generate a group with four elements (the other two elements are indicated by the dashed thick lines).}
	\label{fig:type_B_illustration}
\end{figure}
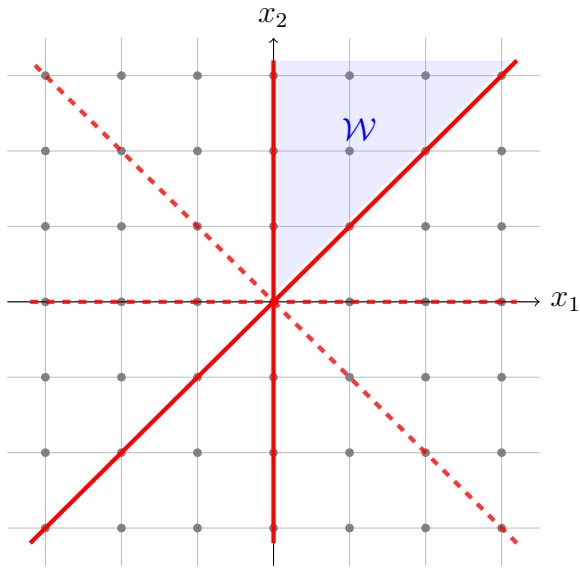
We would like to point out that all results presented in this section have analogues for all general finite or affine reflection groups.
In order to keep this section as short and simple as possible, we restrict our presentation to the type $B_k$ case.
For the general results, we refer the interested reader to the corresponding literature. A good introduction to the theory of reflection groups can be found in the standard reference book by Humphreys~\cite{MR1066460}.

The fundamental assumption underlying this manuscript is the applicability of a reflection principle argument to the problem of counting walks with $n$ composite steps that stay within the region $\mathcal W^0$.
Such a reflection principle has been proved by Gessel and Zeilberger~\cite{MR1092920} for lattice walks in Weyl chambers of arbitrary type that consist of steps from an atomic step set.
We need to slightly extend their result for Weyl chambers of type $B_k$ to walks consisting of steps from a composite step set.
The precise result is stated in the following lemma, and is followed by a short sketch of its proof.
Figure~\ref{fig:reflection_principle} illustrates the idea underlying the proof.
\begin{figure}
	\begin{center}
	\begin{tikzpicture}[scale=.8]
	\draw[lightgray,thin,step=1cm] (-.5,-.5) grid (7.5,7.5);
	\foreach \x in {0,1,...,7} {
		\foreach \y in {0,1,...,7} {
			\draw[gray,fill=gray] (\x,\y) circle (.05cm);
		}
	}
	\draw[->] (-.5,0) -- (7.5,0) node[right] {$x_1$};	
	\draw[->] (0,-.5) -- (0,7.5) node[above] {$x_2$};		
	\draw[white,opacity=.15,fill=blue!50!white] (.05,.15) -- (7.1,7.2) -- (.05,7.2) -- cycle;
	\draw[black,fill=black] (1,2) circle (.15) node[below left] {$\vec u$};
	\draw[black,fill=black] (2,6) circle (.15) node[above right] {$\vec v$};
	\draw[blue,ultra thick,->] (1,2) -- (1,3) -- (2,3);
	\draw[blue,ultra thick,->] (2,3) -- (2,5);
	\draw[blue,ultra thick,->] (2,5) -- (3,5) -- (3,4);
	\draw[blue,ultra thick,->] (3,4) -- (1,4);
	\draw[blue,ultra thick,->] (1,4) -- (1,6);
	\draw[blue,ultra thick,->] (1,6) -- (2,6);
	\scope[xshift=10cm]
	\draw[lightgray,thin,step=1cm] (-.5,-.5) grid (7.5,7.5);
	\foreach \x in {0,1,...,7} {
		\foreach \y in {0,1,...,7} {
			\draw[gray,fill=gray] (\x,\y) circle (.05cm);
		}
	}
	\draw[->] (-.5,0) -- (7.5,0) node[right] {$x_1$};	
	\draw[->] (0,-.5) -- (0,7.5) node[above] {$x_2$};		
	\draw[white,opacity=.15,fill=blue!50!white] (.05,.15) -- (7.1,7.2) -- (.05,7.2) -- cycle;
	\draw[black,fill=black] (1,2) circle (.15) node[below left] {$\vec u$};
	\draw[black,dashed,fill=black] (2,1) circle (.15) node[below left] {$\tau(\vec u)$};
	\draw[black,fill=black] (2,6) circle (.15) node[above right] {$\vec v$};
	\draw[blue,ultra thick,->] (1,2) -- (1,3) -- (2,3);
	\draw[blue,dashed,ultra thick,->] (2,1) -- (3,1) -- (3,2);
	\draw[red,ultra thick,->] (2,3) -- (4,3);
	\draw[red,dashed,ultra thick,->] (3,2) -- (3,3) -- (4,3);
	\draw[blue,ultra thick,->] (4,3) -- (4,4) -- (3,4);
	\draw[blue,ultra thick,->] (3,4) -- (1,4);
	\draw[blue,ultra thick,->] (1,4) -- (1,6);
	\draw[blue,ultra thick,->] (1,6) -- (2,6);
	\endscope
\end{tikzpicture}
	\end{center}
	\caption{The reflection principle. The left picture shows the type of walks we are interested in: walks that are confined to the Weyl chamber $\mathcal W$.
	The right picture illustrates the bijection underlying the proof of Lemma~\ref{lem:reflection_principle}.
}
	\label{fig:reflection_principle}
\end{figure}
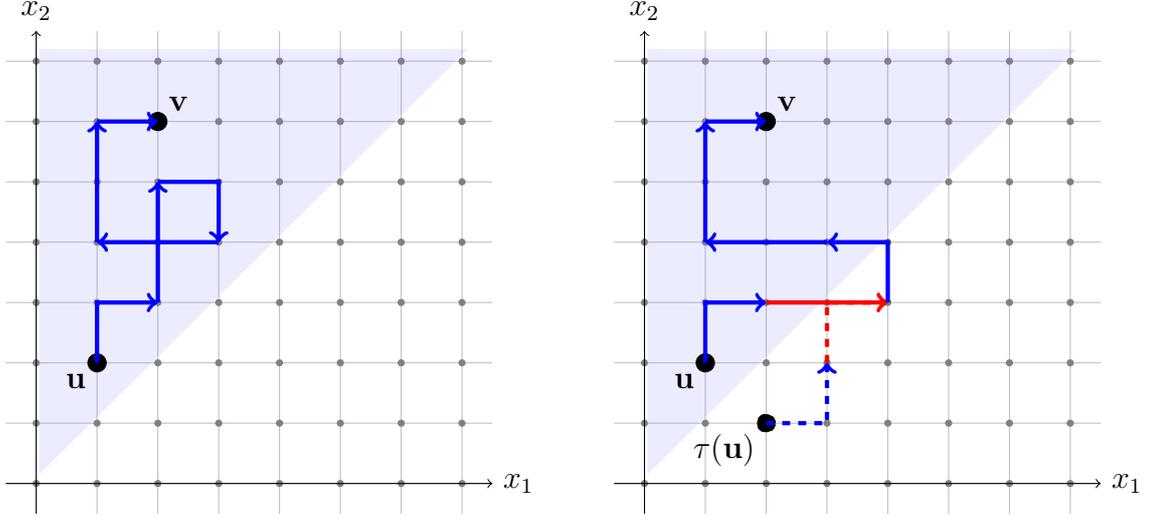

\begin{lemma}[Reflection Principle]\label{lem:reflection_principle}
Let $\mathcal A$ be an atomic step set that is invariant under the reflection group generated by the reflections~\eqref{eq:cox_group_type_b:generators},
and such that for all $\vec a\in\mathcal A$ and all $\vec u\in\mathcal W^0\cap\mathcal L$ we have $\vec u+\vec a\in\mathcal W$.
By $\mathcal S$ we denote a composite step set over $\mathcal A$ such that for all $(\vec a_1,\dots,\vec a_n)\in\mathcal S$ we also have $(\rho(\vec a_1),\dots,\rho(\vec a_j),\vec a_{j+1},\dots,\vec a_m)\in\mathcal S$ for all $j=1,2,\dots,m$ and all reflections $\rho$ in the group generated by \eqref{eq:cox_group_type_b:generators}.
Finally, assume that the weight function $w:\mathcal S\to\R_+$ satisfies $w\left( (\vec a_1,\dots,\vec a_m) \right)=w\left( (\rho(\vec a_1),\dots,\rho(\vec a_j),\vec a_{j+1},\dots,\vec a_m) \right)$ for all $j$ and $\rho$ as before.

%
Then, for all $\vec u=(u_1,\dots,u_k)\in\mathcal W^0\cap\mathcal L$ and all $\vec v\in\mathcal W^0\cap\mathcal L$, the generating function for all $n$-step walks with steps from the composite step set $\mathcal S$ with respect to the weight $w$ that stay within $\mathcal W^0$ satisfies
	\begin{equation}\label{eq:reflection_principle_type_b}
		P_n^+(\vec u\to\vec v)=\sum_{\substack{\sigma\in\mathfrak S_k \\ \varepsilon_1,\dots,\varepsilon_k\in\left\{ -1,+1 \right\}}}
		\left( \prod_{j=1}^k\varepsilon_j \right)\sgn(\sigma)P_n\Big( (\varepsilon_1 u_{\sigma(1)},\dots,\varepsilon_k u_{\sigma(k)})\to\vec v \Big),
	\end{equation}
	where $\mathfrak S_k$ is the set of all permutations on $\left\{ 1,\dots,k \right\}$.
\end{lemma}
\begin{proof}[Proof (Sketch)]
The proof of this lemma is almost identical to the proof of the reflection principle for lattice walks consisting of atomic steps in \cite{MR1092920}.
The basic idea of the proof is the following.
We set up an involution on the set of $n$-step walks starting in one of the points$(\rho(a_1),\dots,\rho(a_k))$, where $\rho$
denotes an arbitrary reflection in the group generated by \eqref{eq:cox_group_type_b:generators}, to $\vec v$ that percolate or touch the boundary of $\mathcal W$.
For a typical such walk we then show that the contributions of it and its image under this involution to the right hand side of \eqref{eq:reflection_principle_type_b} differ by sign only.
This shows that the total contribution of $n$-step walks percolating or touching the boundary of $\mathcal W$ to the right hand side of \eqref{eq:reflection_principle_type_b} is equal to zero.

This involution is constructed with the help of the involution (for an illustration, see Figure~\ref{fig:reflection_principle}) defined in the proof of \cite[Theorem 1]{MR1092920} as follows.
Consider the walk starting in $(\rho(u_1),\dots,\rho(u_k))$ with step sequence
\[ 
\big( (\vec a_{1,1},\dots,\vec a_{1,m_j}),(\vec a_{2,1},\dots,\vec a_{2,m_2}),\cdots,(\vec a_{n,1},\dots,\vec a_{n,m_n})\big)\in\mathcal S^n,
\]
where the $\vec a_{j,\ell}$ denote atomic steps.
If we ignore all the inner brackets in the step sequence above, we can view this walk as a walk consisting of $(m_1+\dots+m_n)$ atomic steps.
To this walk, we can apply the involution of the proof of \cite[Theorem 1]{MR1092920}.

For example, assume that the first contact of the walker with the boundary of $\mathcal W$ occurs right after the atomic step $\vec a_{j,\ell}$.
Then, the image of this path under the involution is the path starting in $(\tau(\rho(u_1)),\dots,\tau(\rho(u_k)))$ with step sequence
\[
\big( (\tau(\vec a_{1,1}),\dots,\tau(\vec a_{1,m_1})),\cdots,
(\tau(\vec a_{j,1}),\dots,\tau(\vec a_{j,\ell}),\vec a_{j,\ell+1},\dots,\vec a_{j,m_j}),\cdots,
(\vec a_{n,1},\dots,\vec a_{n,m_n}) \big),
\]
for a specifically chosen reflection $\tau$ in one of the hyperplanes~\eqref{eq:cox_group_type_b:generators}.

For a details, we refer the reader to the proof of \cite[Theorem 1]{MR1092920}.
\end{proof}
%


In view of this last lemma, the question that now arises is: what composite step sets $\mathcal S$ satisfy the conditions in Lemma~\ref{lem:reflection_principle}?
This question boils down the question: what atomic step sets $\mathcal A$ satisfy the conditions in Lemma~\ref{lem:reflection_principle}?
The answer to this latter question has been given by Grabiner and Magyar~\cite{MR1235279}.
For type $B$, the result reads as follows.

\begin{lemma}[Grabiner and Magyar~\cite{MR1235279}]\label{lem:grabiner:classification}
The atomic step set $\mathcal A\subset\R^k\setminus\left\{ \vec 0 \right\}$ satisfies the conditions stated in Lemma~\ref{lem:reflection_principle} if and only if $\mathcal A$ is (up to rescaling) equal either  to
\[
\left\{ \pm\vec e^{(1)},\pm\vec e^{(2)},\dots,\pm\vec e^{(k)} \right\}
\qquad\textrm{or to}\qquad
\left\{ \sum_{j=1}^{k}\varepsilon_j\vec e^{(j)}\ :\ \varepsilon_1,\dots,\varepsilon_k\in\left\{ -1,+1 \right\} \right\}, \]
where $\left\{ \vec e^{(1)},\dots,\vec e^{(k)} \right\}$ is the canonical basis in $\R^k$.
\end{lemma}

In this manuscript we will always assume that our lattice walk model satisfies all the requirements of Lemma~\ref{lem:reflection_principle}.
Therefore, we make the following assumption.
\begin{assumption}\label{ass:step_sets}
From now on, we assume that the atomic step set $\mathcal A$ is equal to one of the two sets given in Lemma~\ref{lem:grabiner:classification}.
Further, we assume that the composite step set $\mathcal S$ and the weight function $w:\mathcal S\to\R^k$ satisfy the conditions of Lemma~\ref{lem:reflection_principle}.
\end{assumption}

The final objective in this section is an integral formula for $P_n^+(\vec u\to\vec v)$.
The result is stated in Lemma~\ref{lem:exact_integral_u->v} below.
Its derivation is based on a generating function approach.

In order to simplify the presentation, we apply the standard multi-index notation:
If $\vec z=(z_1,\dots,z_k)$ is a vector of indeterminates and $\vec a=(a_1,\dots,a_k)\in\Z^k$, then we set $\vec z^{\vec a}:=z_1^{a_1}z_2^{a_2}\dots z_k^{a_k}$.
Furthermore, if $F(\vec z)$ is a series in $\vec z$, then we denote by $[\vec z^{\vec a}]F(\vec z)$ the coefficient of the monomial $\vec z^{\vec a}$ in  $F(\vec z)$.

Now, we define the \emph{atomic step generating function} $A(\vec z)=A(z_1,\dots,z_k)$ associated with the atomic step set $\mathcal A$ by
\[ A(z_1,\dots,z_k)=A(\vec z)=\sum_{\vec a\in\mathcal A}\vec z^{\vec a}. \]
The \emph{composite step generating function} associated with the composite step set $\mathcal S$ with respect to the weight $w$ is defined by
\[ S(z_1,\dots,z_k)=S(\vec z)=\sum_{\substack{m\ge 0 \\ (\vec a_1,\dots,\vec a_m)\in\mathcal S}}w\Big( (\vec a_1,\dots,\vec a_m) \Big)\vec z^{\vec a_1+\dots+\vec a_m}. \]

The generating function for the number of $n$-step paths with steps from the composite step set $\mathcal S$ that start in $\vec u\in\mathcal L$ and end in $\vec v\in\mathcal L$ with respect to the weight $w$ can then be expressed as
\begin{equation}\label{eq:paths_u->v_coefficient_repr}
P_n(\vec u\to\vec v)=\left[ \vec z^{\vec v-\vec u} \right]S(\vec z)^n.
\end{equation}

We can now state and prove the main result of this section: the integral formula for $P_n^+(\vec u\to\vec v)$.
\begin{lemma}\label{lem:exact_integral_u->v}
	Let $\mathcal S$ be a composite step set and let $w:\mathcal S\to\R_+$ be weight function, both satisfying Assumption~\ref{ass:step_sets}.
	Furthermore, let $S(z_1,\dots,z_k)$ be the associated composite step generating function.
	
	Then the generating function $P_n^+(\vec u\to\vec v)$ for the number of $n$-step paths from $\vec u\in\mathcal W^0\cap\mathcal L$
	to $\vec v\in\mathcal W^0\cap\mathcal L$ that stay within $\mathcal W^0$ with steps from the composite step set $S$ satisfies
	\begin{equation}\label{eq:exact_integral_u->v}
	P_{n}^+(\vec u\to\vec v)=
	\frac{1}{(2\pi i)^k}\idotsint\limits_{|z_1|=\dots=|z_k|=\rho}\det_{1\le j,m\le k}\left( z_j^{u_m}-z_j^{-u_m} \right)S(z_1,\dots,z_k)^n\left( \prod_{j=1}^{k}\frac{dz_j}{z_j^{v_j+1}}\right),
\end{equation}
where $\rho>0$.
\end{lemma}
\begin{proof}
	The proof of this lemma relies on the reflection principle (Lemma~\ref{lem:reflection_principle}) and Cauchy's integral formula.

	Lemma~\ref{lem:reflection_principle} and Equation~\eqref{eq:paths_u->v_coefficient_repr} together give us
	\[
	P_{n}^+(\vec u\to\vec v)=
	\sum_{\substack{\sigma\in\mathfrak S_k \\ (\varepsilon_1,\dots,\varepsilon_k)\in\left\{ -1,+1 \right\}^k}}
	\left( \prod_{j=1}^{k}\varepsilon_j \right)\sgn(\sigma)\left[ z_1^{v_1-\varepsilon_1 u_{\sigma(1)}}\dots z_k^{v_k-\varepsilon_k u_{\sigma(k)}}\right]
	S(z_1,\dots,z_k)^n,
	\]
	and by Cauchy's integral formula, we have
	\begin{multline*}
	\left[ z_1^{v_1-\varepsilon_1 u_{\sigma(1)}}\dots z_k^{v_k-\varepsilon_k u_{\sigma(k)}}\right]S(z_1,\dots,z_k)^n \\
	=
	\frac{1}{(2\pi i)^k}\idotsint\limits_{|z_1|=\dots=|z_k|=1}S(z_1,\dots,z_k)^n\left( \prod_{j=1}^{k}\frac{dz_j}{z_j^{v_j-\varepsilon_k u_{\sigma(j)}+1}} \right).
	\end{multline*}
	Now, substituting the right hand side of the last equation above for the corresponding term in the second to last equation, and interchanging summation and integration, we obtain the expression
	\[
	\idotsint\limits_{|z_1|=\dots=|z_k|=1}\frac{S(z_1,\dots,z_k)^n}{(2\pi i)^k}
	\left(
	\sum_{\substack{\sigma\in\mathfrak S_k \\ (\varepsilon_1,\dots,\varepsilon_k)\in\left\{ -1,+1 \right\}^k}}
	\sgn(\sigma)\left( \prod_{j=1}^k\varepsilon_j z_j^{\varepsilon_ju_{\sigma(j)}} \right)
	\right)
	\left( \prod_{j=1}^{k}\frac{dz_j}{z_j^{v_j+1}} \right).
	\]
	The result now follows from this expression by noting that
	\[
	\sum_{\substack{\sigma\in\mathfrak S_k \\ (\varepsilon_1,\dots,\varepsilon_k)\in\left\{ -1,+1 \right\}^k}}
	\left( \prod_{j=1}^{k}\varepsilon_j \right)\sgn(\sigma)
	\left( \prod_{j=1}^{k}z_j^{\varepsilon_j u_{\sigma(j)}} \right)
	=\det_{1\le j,m\le k}\left( z_j^{u_m}-z_j^{-u_m} \right).
	\]
\end{proof}

We close this section with two alternative exact expression for the quantity $P_n^+(\vec u\to\vec v)$ that will become important later on.
\begin{corollary}\label{cor:exact_integral_u->v_alternative}
	Under the conditions of Lemma~\ref{lem:exact_integral_u->v}, the generating function $P_n^+(\vec u\to\vec v)$ for the number of $n$-step paths from $\vec u\in\mathcal W^0\cap\mathcal L$ to $\vec v\in\mathcal W^0\cap\mathcal L$ that stay within $\mathcal W^0$ with steps from the composite step set $S$ satisfies
\begin{multline*}
	P_{n}^+(\vec u\to\vec v)=
	\frac{(-1)^k}{(2\pi i)^kk!}\\ \times\idotsint\limits_{|z_1|=\dots=|z_k|=\rho}\det_{1\le j,m\le k}\left( z_j^{u_m}-z_j^{-u_m} \right)S(z_1,\dots,z_k)^n\det_{1\le j,m\le k}\left( z_j^{v_m} \right)\left( \prod_{j=1}^{k}\frac{dz_j}{z_j}\right), 
\end{multline*}
where $\rho>0$.
\end{corollary}
\begin{proof}
	The substitution $z_j\mapsto 1/z_j$, for $j=1,2,\dots,k$, transforms Equation~\eqref{eq:exact_integral_u->v} into
\begin{equation*}
	P_{n}^+(\vec u\to\vec v)=
	\frac{(-1)^k}{(2\pi i)^k}\idotsint\limits_{|z_1|=\dots=|z_k|=\rho}\det_{1\le j,m\le k}\left( z_j^{u_m}-z_j^{-u_m} \right)S(z_1,\dots,z_k)^n\left( \prod_{j=1}^{k}z_j^{v_j}\frac{dz_j}{z_j}\right). 
\end{equation*}
Now, we make the following observation.
If $\sigma$ denotes an arbitrary permutation on the set $\left\{ 1,2,\dots,k \right\}$, then we have
\[
\det_{1\le j,m\le k}\left( z_{\sigma(j)}^{u_m}-z_{\sigma(j)}^{-u_m} \right)\left( \prod_{j=1}^kz_{\sigma(j)}^{v_m} \right)=
\det_{1\le j,m\le k}\left( z_j^{u_m}-z_j^{-u_m} \right)\left(\sgn(\sigma) \prod_{j=1}^kz_{\sigma(j)}^{v_m} \right),
\]
which can be seen to be true by rearranging the rows of the determinant on the left hand side and taking into account the sign changes.
This implies
\begin{multline*}
	P_{n}^+(\vec u\to\vec v)=
	\frac{(-1)^k}{(2\pi i)^k}\\ \times\idotsint\limits_{|z_1|=\dots=|z_k|=\rho}\det_{1\le j,m\le k}\left( z_j^{u_m}-z_j^{-u_m} \right)S(z_1,\dots,z_k)^n\left(\sgn(\sigma) \prod_{j=1}^{k}z_{\sigma(j)}^{v_j}\frac{dz_j}{z_j}\right). 
\end{multline*}
The claim is now proved upon
summing the expression above over all $k!$ possible permutations and dividing the result by $k!$. 
\end{proof}

\begin{corollary}\label{cor:exact_integral_u->v_alternative2}
	Under the conditions of Lemma~\ref{lem:exact_integral_u->v}, the generating function $P_n^+(\vec u\to\vec v)$ for the number of $n$-step paths from $\vec u\in\mathcal W^0\cap\mathcal L$ to $\vec v\in\mathcal W^0\cap\mathcal L$ that stay within $\mathcal W^0$ with steps from the composite step set $S$ satisfies
\begin{multline*}
	P_{n}^+(\vec u\to\vec v)=
	\frac{1}{k!}\left(\frac{i}{4\pi}\right)^k\\ \times\idotsint\limits_{|z_1|=\dots=|z_k|=\rho}\det_{1\le j,m\le k}\left( z_j^{u_m}-z_j^{-u_m} \right)\det_{1\le j,m\le k}\left( z_j^{v_m}-z_j^{-v_m} \right)S(z_1,\dots,z_k)^n\left( \prod_{j=1}^{k}\frac{dz_j}{z_j}\right), 
\end{multline*}
where $\rho>0$.
\end{corollary}
\begin{proof}[Proof (Sketch)]
	The representation can be proved in almost the same way as Corollary~\ref{cor:exact_integral_u->v_alternative} except that now we let the group $B_k$ act on the set of integration variables instead of the symmetric group $S_k$.
	Indeed, for $\sigma\in S_k$ and $\varepsilon_1,\dots,\varepsilon_k\in\left\{ \pm 1 \right\}$, we find
\[
\det_{1\le j,m\le k}\left( z_{\sigma(j)}^{\varepsilon_ju_m}-z_{\sigma(j)}^{-\varepsilon_ju_m} \right)\left( \prod_{j=1}^kz_{\sigma(j)}^{\varepsilon_jv_m} \right)=
\det_{1\le j,m\le k}\left( z_j^{u_m}-z_j^{-u_m} \right)\left(\sgn(\sigma) \prod_{j=1}^k\varepsilon_jz_{\sigma(j)}^{\varepsilon_jv_m} \right),
\]
and summing the right hand side over all possible values for $\varepsilon_1,\dots,\varepsilon_k$ yields
\[
\det_{1\le j,m\le k}\left( z_j^{u_m}-z_j^{-u_m} \right)\left(\sgn(\sigma) \prod_{j=1}^k\left( z_{\sigma(j)}^{v_m}-z_{\sigma(j)}^{-v_m}\right)\right).
\]
Now, summing over all possible $\sigma\in \Sigma_k$ and dividing by $2^kk!$ (the cardinality of $B_k$), we obtain the result.
\end{proof}

\section{Step generating functions}\label{sec:auxiliary}

%
The proofs of Theorem~\ref{thm:asymptotic_u->v} and Theorem~\ref{thm:asymptotic_u->} rely on some structural results for composite step generating functions $S(z_1,\dots,z_k)$ associated with composite step sets that satisfy Assumption~\ref{ass:step_sets} (the conditions of Lemma~\ref{lem:reflection_principle}).
These structural results are the content of the present section.

A direct consequence of the classification of Grabiner and Magyar~\cite{MR1235279}, presented in Lemma~\ref{lem:grabiner:classification},
is the following result on atomic step generating functions.
\begin{lemma}\label{lem:atomic_step_gen_functions}
	Let $\mathcal A$ be an atomic step set satisfying Assumption~\ref{ass:step_sets}.
	Then the associated atomic step generating function $A(z_1,\dots,z_k)$ is equal either to
	\begin{equation}\label{eq:possible_atomic_step_gen_fun}
		\sum_{j=1}^{k}\left( z_j+\frac{1}{z_j}\right)
		\qquad\textrm{or to}\qquad
		\prod_{j=1}^{k}\left( z_j+\frac{1}{z_j}\right).
	\end{equation}
\end{lemma}

As a direct consequence of this last lemma, we obtain the following result.
\begin{lemma}\label{lem:step_gen_fun_structure_typeC}
	Let $\mathcal S$ be composite step set over the atomic step set $\mathcal A$, and let $w:\mathcal S\to\R_+$ be a weight function.
	If $\mathcal S$, $\mathcal A$ and $w$ satisfy Assumption~\ref{ass:step_sets}, then there exists a polynomial $P(x)$ with non-negative coefficients such that either
	\[
	S(z_1,\dots,z_k) = P\left( \sum_{j=1}^{k}\left( z_j+\frac{1}{z_j}\right)\right)
	\qquad\textrm{or}\qquad
	S(z_1,\dots,z_k) = P\left( \prod_{j=1}^{k}\left( z_j+\frac{1}{z_j}\right)\right).
	\]
\end{lemma}
\begin{proof}
	Let $A(z_1,\dots,z_k)$ denote the atomic step generating function corresponding to $\mathcal A$.

	Our assumptions imply that if $(\vec a_1,\dots,\vec a_m)\in\mathcal S$, then we also have
	\[ 
	(\rho(\vec a_1),\dots,\rho(\vec a_j),\vec a_{j+1},\dots,\vec a_m)\in\mathcal S,
	\qquad j=1,2,\dots,m
	\]
	and all reflections $\rho$ in the group generated by \eqref{eq:cox_group_type_b:generators}. 
	This means that if the composite step set $\mathcal S$ contains a composite step consisting of $m$ atomic steps, then $\mathcal S$ has to contain \emph{all} composite steps consisting of $m$ atomic steps. 
	Also, our assumptions on $w$ imply that the same weight is assigned to all composite steps consisting of the same number of atomic steps.
	Since the generating function for all composite steps consisting of $m$ atomic steps is given by $A(z_1,\dots,z_k)^m$, we deduce that there exists a polynomial $P(x)$ with non-negative coefficients such that $S(z_1,\dots,z_k)=P(A(z_1,\dots,z_k))$.
	This fact, together with Lemma~\ref{lem:atomic_step_gen_functions}, proves the claim.
\end{proof}

Crucial to our analysis will be information on the maxima of $|S(z_1,\dots,z_k)|$ on the unit torus.
In order to simplify the presentation, we give the following definition.
\begin{definition}
	Let $S(z_1,\dots,z_k)$ be a composite step generating function.
	A point $\hat\varphi=(\hat\varphi_1,\cdots,\hat\varphi_k)$ is called a \emph{maximal point} if and only if
	it is a maximum of the function
	\[ (\varphi_1,\cdots,\varphi_k)\mapsto |S(e^{i\varphi_1},\cdots,e^{i\varphi_k})|. \]

	The set of maximal points will always be denoted by $\mathcal M$.
\end{definition}

We proceed with further consequences that result from the structural result in Lemma~\ref{lem:step_gen_fun_structure_typeC} that will become useful later on.
As a first step, we need to have a closer look on the set $\mathcal M$ of maximal points.

\begin{lemma}\label{lem:step_gen_fun_maxima}
	Let $S(z_1,\dots,z_k)$ be a composite step generating function.
	Then, the set $\mathcal M$ of maximal points associated with $S(z_1,\dots,z_k)$ satisfies the inclusions
	\[ \left\{ (0,\dots,0) \right\}\subseteq\mathcal M\subseteq\left\{ 0,\pi \right\}^k. \]
%
\end{lemma}
\begin{proof}
	From Lemma~\ref{lem:step_gen_fun_structure_typeC}, we deduce that $S(e^{i\varphi_1},\dots,e^{i\varphi_k})$ is either equal to
\begin{equation}\label{eq:step_gen_fun_torus_repr}
	P\left( 2\sum_{j=1}^{k}\cos\varphi_j\right)
	\qquad\textrm{or to}\qquad
	P\left( 2^k\prod_{j=1}^{k}\cos\varphi_j\right),
\end{equation}
for some polynomial $P(x)$ with non-negative coefficients.
Now, if $S(e^{i\varphi_1},\dots,e^{i\varphi_k})$ is equal to the expression on the left in \eqref{eq:step_gen_fun_torus_repr}, then the triangle inequality shows that
\[
 |S(e^{i\varphi_1},\dots,e^{i\varphi_k})|=\left|P\left( 2\sum_{j=1}^k\cos(\varphi_j) \right)\right|
 \le P\left( 2\sum_{j=1}^k\left|\cos(\varphi_j)\right| \right)\le S(1,\dots,1).
\]
If $S(e^{i\varphi_1},\dots,e^{i\varphi_k})$ is equal to the expression on the right in \eqref{eq:step_gen_fun_torus_repr}, then similar arguments can be used to show the inequality $|S(e^{i\varphi_1},\dots,e^{i\varphi_k})|\le S(1,\dots,1)$ in this case.
This inequality shows that $(0,\dots,0)$ is always a maximum of the function $(\varphi_1,\dots,\varphi_k)\mapsto |S(e^{i\varphi_1},\dots,e^{i\varphi_k})|$,
and further, since $P(x)$ is monotonic increasing for $x>0$, that all points maximising this function lie within the set $\left\{ 0,\pi \right\}^k$.
%
\end{proof}

\begin{remark}
	The possible situations can be a bit more accurately described by distinguishing the two cases listed in Lemma~\ref{lem:step_gen_fun_structure_typeC}.
	If $S(z_1,\dots,z_k)=P(A(z_1,\dots,z_k))$, then one of the following situations hold true.
	\begin{enumerate}
		\item If $A(z_1,\dots,z_k)=\sum_{j=1}^k(z_j+z_j^{-1})$, then we have $\mathcal M\subseteq\left\{ (0,\dots,0),(\pi,\dots,\pi) \right\}$.
			Moreover, if the polynomial $P(x)$ is neither even nor odd, i.e., $P(-x)\neq \pm P(x)$, then $\mathcal M=\left\{ (0,\dots,0) \right\}$.
		\item If $A(z_1,\dots,z_k)=\prod_{j=1}^k(z_j+z_j^{-1})$, then each lattice point $\vec u=(u_1,\dots,u_k)\in\mathcal L$ satisfies $u_1\equiv\dots\equiv u_k\mod 2$. 
			Moreover, if $P(-x)\neq \pm P(x)$, then
			\[ \mathcal M\subseteq \left\{ (\hat\varphi_1,\dots,\hat\varphi_k)\in\left\{ 0,\pi \right\}^k\ :\ \sum_{j=1}^k\frac{\hat\varphi_j}{\pi}\equiv 0\mod 2 \right\}. \]
	\end{enumerate}
	\label{rem:situations_casebycase}
\end{remark}

\begin{lemma}
	For any maximal point $\hat\varphi=(\hat{\varphi}_1,\dots,\hat\varphi_k)\in\mathcal M$ we have
	\[
	S\left( e^{i(\hat\varphi_1+\varphi)},\dots,e^{i(\hat\varphi_k+\varphi_k)} \right)=
	\frac{S\left( e^{i\hat\varphi_1},\dots,e^{i\hat\varphi_k} \right)}{S(1,\dots,1)}
	S\left( e^{i\varphi_1},\dots,e^{i\varphi_k} \right).
	\]
	\label{lem:step_gen_fun_1}
\end{lemma}
\begin{proof}
	By Lemma~\ref{lem:step_gen_fun_structure_typeC}, we know that the only possible step generating functions are given by \eqref{eq:step_gen_fun_torus_repr}.
	Now, if $\vec{\hat\varphi}\in\mathcal M$ is a maximal point, then both expressions
	\[
	\frac{\sum_{j=1}^k\cos\left( \hat\varphi_j+\varphi_j \right)}{\sum_{j=1}^k\cos\left( \varphi_j \right)}
	\qquad\textrm{and}\qquad
	\frac{\prod_{j=1}^k\cos\left( \hat\varphi_j+\varphi_j \right)}{\prod_{j=1}^k\cos\left( \varphi_j \right)}
	\]
	are constants, and more precisely either equal to $+1$ or to $-1$.
	Now, if $P(x)$ is neither even nor odd, then nothing is to prove as in that case we must have $+1$ and $S\left( e^{i\hat\varphi_1},\dots,e^{i\hat\varphi_k} \right)=S(1,\dots,1)$.
	Otherwise, the constant is given by $S\left( e^{i\hat\varphi_1},\dots,e^{i\hat\varphi_k} \right)/S(1,\dots,1)$, which proves the claim.
\end{proof}

\begin{lemma}
We have the asymptotics
\begin{multline*}
\log \left|S\left( e^{i\varphi_1},\dots,e^{i\varphi_k} \right)\right|
=
\log\left( S(1,\dots,1) \right)-\Lambda\sum_{j=1}^k\frac{\varphi_j^2}{2}+\frac{\Omega}{2}\left( \sum_{j=1}^k\varphi_j^2 \right)^2+\Psi\sum_{j=1}^k\frac{\varphi_j^4}{4!}\\+O\left( \max_{j}|\varphi_j|^6 \right)
\end{multline*}

as $\max_j|\varphi_j|\to 0$, where either 
\[
\Lambda=2\frac{P'(2k)}{P(2k)},\quad
\Omega=4\frac{P''(2k)}{P(2k)^2}-\Lambda^2\quad\textrm{and}\quad
\Psi=\Lambda
\]
or
\[
\Lambda=2^k\frac{P'(2^k)}{P(2^k)},\quad
\Omega=4^k\frac{P''(2k)}{P(2k)^2}-\Lambda^2+\Lambda\quad\textrm{and}\quad
\Psi=-2\Lambda,
\]
corresponding to the two cases in Lemma~\ref{lem:step_gen_fun_structure_typeC}.
Here, $P'(x)$ and $P''(x)$ denote the first and second derivative with respect to $x$.
\label{lem:step_gen_fun_asymptotics}
\end{lemma}

\begin{lemma}\label{lem:step_gen_fun_det_identity}
	For any maximal point $\hat\varphi=(\hat\varphi_1,\dots,\hat\varphi_k)\in\mathcal M$ and any lattice point $\vec u=(u_1,\dots,u_k)\in\mathcal L$ we have
	\begin{equation}\label{eq:free:1}
\det_{1\le j,m\le k}\Big( \sin\left(u_m(\hat\varphi_j+\varphi_j)\right) \Big)
		=
		(-1)^{\sum_{j=1}^ku_j\hat\varphi_j/\pi}\det_{1\le j,m\le k}\Big(\sin\left( u_m\varphi_j \right)\Big)
	\end{equation}
\end{lemma}
\begin{proof}
	As a first step we note that
	\[
\det_{1\le j,m\le k}\Big( \sin\left(u_m(\hat\varphi_j+\varphi_j)\right) \Big)
=
\det_{1\le j,m\le k}\Big( (-1)^{u_m\hat\varphi_j/\pi}\sin\left(u_m\varphi_j\right) \Big).
	\]
	Now, the result follows from a case by case analysis.
	If the composite step generating function satisfies $S(z_1,\dots,z_k)=P\left(\sum_{j=1}^k\left( z_j+\frac{1}{z_j} \right)\right)$ for some non-negative polynomial $P$, then we have at most two maximal points, namely $(0,\dots,0)$ and $(\pi,\dots,\pi)$ and the result follows.

	If, on the other hand, we have $S(z_1,\dots,z_k)=P\left(\prod_{j=1}^k\left( z_j+\frac{1}{z_j} \right)\right)$, then we know that
	$u_1\equiv\dots\equiv u_k\ \mod 2$ and the result follows.
\end{proof}

\section{Determinants and asymptotics}\label{sec:determinants_and_asymptotics}
Asymptotics for determinants are often hard to obtain, the reason being a typical large number of cancellations of asymptotically leading terms. 
In this section, we present a factorisation technique that allows one to represent certain functions in several complex variables defined by determinants as a product of two factors.
One of these factors will always be a symmetric (Laurent) polynomial (this accounts for the cancellations of asymptotically leading terms mentioned before).
The second factor is a determinant, the entries of which are certain contour integrals.
In many cases, asymptotics for this second factor can be established geometric series expansions, coefficient extraction and known determinant evaluations.
The fundamental technique is illustrated in Lemma~\ref{lem:det_factorisation} below.

We want to stress that Lemma~\ref{lem:det_factorisation} should be seen as a general technique, not as a particular result.
The main intention of this lemma is to give the reader an unblurred view at the technique.
An application of Lemma~\ref{lem:det_factorisation} together with some remarks on asymptotics can be found right after the proof.

%

Let us now start with the illustration of our factorisation technique.
\begin{lemma}\label{lem:det_factorisation}
	Let $A_m(x,y)$, $1\le m\le k$, be analytic and one-valued for $(x,y)\in\mathcal R\times\mathcal D\subset\C^2$, where $\mathcal D\subset\C$ is some non empty set and $\mathcal R=\left\{ x\in\C\ :\ r^\ast\le |x|<R^\ast \right\}$ for some $0\le r^\ast<R^\ast$.

Then, the function
\[ \det_{1\le j,m\le k}\left( A_m(x_j,y_m) \right) \]
is analytic for $(x_1,\dots,x_k,y_1,\dots,y_k)\in\mathcal R^k\times\mathcal D^k$, and it satisfies
\begin{multline*}
\det_{1\le j,m\le k}\left( A_m(x_j,y_m) \right)
=
\left( \prod_{1\le j<m\le k}(x_m-x_j) \right)\\
\times\det_{1\le j,m\le k}\left( \frac{1}{2\pi i}\int\limits_{|\xi|=R}\frac{A_m(\xi,y_m)d\xi}{\prod\limits_{\ell=1}^j(\xi-x_\ell)} - \frac{1}{2\pi i}\int\limits_{|\xi|=r}\frac{A_m(\xi,y_m)d\xi}{\prod\limits_{\ell=1}^j(\xi-x_\ell)} \right),
\end{multline*}
where $r^\ast<r<\min_{j}|x_j|\le\max_{j}|x_j|<R<R^\ast$.
\end{lemma}
\begin{proof}
By Laurent's theorem, we have
\begin{equation}\label{eq:dominant_term_det+laurent}
\det_{1\le j,m\le k}\left( A_m(x_j,y_m) \right)
=
\det_{1\le j,m\le k}\left( \frac{1}{2\pi i}\int\limits_{|\xi|=R}\frac{A_m(\xi,y_m)d\xi}{\xi-x_j} - \frac{1}{2\pi i}\int\limits_{|\xi|=r}\frac{A_m(\xi,y_m)d\xi}{\xi-x_j}  \right).
\end{equation}
Now, short calculations show that for any $L\ge 0$ and all $n_1,\dots,n_L\in\left\{ 1,2,\dots,k \right\}$ we have
\begin{multline*}
	\int_{|\xi|=\rho_1}\frac{A_m(\xi,y_m)d\xi}{(\xi-x_j)\prod_{\ell=1}^L(\xi-x_{n_\ell})}-\int_{|\xi|=\rho_1}\frac{A_m(\xi,y_m)d\xi}{(\xi-x_j)\prod_{\ell=1}^L(\xi-x_{n_\ell})} \\
	=(x_m-x_j)\int_{|\xi|=\rho_1}\frac{A(\xi,y)d\xi}{(\xi-x_j)(\xi-x_m)\prod_{\ell=1}^L(\xi-x_{n_\ell})}.
\end{multline*}
Consequently, we can prove the claimed factorisation as follows.
First, we subtract the first row of the determinant in~\eqref{eq:dominant_term_det+laurent} from all other rows.
By the computations above we can then take the factor $(x_j-x_1)$ out of the $j$-th row of the determinant.
In a second run, we subtract the second row from the rows $3,4,\dots,k$, and so on.
In general, after subtracting row $j$ from row $\ell$ we take the factor $(x_{\ell}-x_j)$ out of the determinant.
\end{proof}

\begin{example}\label{ex:factorisation_technique}
Consider the function 
\[ \det_{1\le j,m\le k}\left( e^{x_jy_m} \right). \]
An application of Lemma~\ref{lem:det_factorisation} with $A(x,y)=e^{xy}$ immediately gives us the factorisation
\[
\det_{1\le j,m\le k}\left( e^{x_jy_m} \right)
=
\left( \prod_{1\le j<m\le k}(x_m-x_j) \right)
\det_{1\le j,m\le k}\left( \frac{1}{2\pi i}\int_{|\xi|=R}\frac{e^{\xi y_m}d\xi}{\prod_{\ell=1}^j(\xi-x_\ell)} \right),
\]
where $R>\max_j|x_j|$.
Note that the second contour integral occurring in the factorisation given in Lemma~\ref{lem:det_factorisation} is equal to zero because the function $A(x,y)=e^{xy}$ is an entire function.

Now we want to demonstrate how one can derive asymptotics for $\det_{1\le j,m\le k}\left( e^{x_jy_m} \right)$ as $x_1,\dots,x_k\to 0$ from this factorisation.
The geometric series expansion gives us
\begin{align*}
\frac{1}{2\pi i}\int_{|\xi|=R}\frac{e^{\xi y}d\xi}{\prod_{\ell=1}^j(\xi-x_\ell)}
&=\frac{1}{2\pi i}\int_{|\xi|=R}e^{\xi y}\frac{d\xi}{\xi^j}+O\left( \sum_{j=1}^k|x_k| \right) \\
&=\frac{y^{j-1}}{(j-1)!}+O\left( \sum_{j=1}^k|x_k| \right)
\end{align*}
as $x_1,\dots,x_k\to 0$.
Consequently, we have
\begin{align*}
\det_{1\le j,m\le k}\left( e^{x_jy_m} \right)
&=\left( \prod_{1\le j<m\le k}(x_m-x_j) \right)
\left( \det_{1\le j,m\le k}\left( \frac{y_m^{j-1}}{(j-1)!} \right)+O\left( \sum_{j=1}^k|x_k| \right) \right) \\
&=\left( \prod_{1\le j<m\le k}(x_m-x_j) \right)
\left( \left( \prod_{1\le j<m\le k}\frac{y_m-y_j}{m-j} \right)+O\left( \sum_{j=1}^k|x_j| \right) \right)
\end{align*}
as $x_1,\dots,x_k\to 0$.

This illustrates that the problem of establishing asymptotics for functions of the form $\det_{1\le j,m\le k}(A_m(x_j,y_m))$ can be reduced to an application of Lemma~\ref{lem:det_factorisation} and the extraction of certain coefficients of the functions $A_m(x,y)$.
\end{example}

\begin{example}
If we would have considered the function $\det\limits_{1\le j,m\le k}\left( e^{\xi^2y} \right)$, $k>1$, instead of $\det\limits_{1\le j,m\le k}(e^{x_jy_m})$ as in the example above, we would have got only the upper bound
\[
\det_{1\le j,m\le k}\left( e^{x_j^2y_m} \right)
=O\left( \left( \prod_{1\le j<m\le k}(x_m-x_j) \right)\sum_{j=1}^k|x_j| \right)
\]
as $x_1,\dots,x_k\to 0$, because
\[
\det_{1\le j,m\le k}\left( \frac{1}{2\pi i}\int_{|\xi|=R}e^{\xi^2y_m}\frac{d\xi}{\xi^j} \right)=0,\qquad k>1.
\]
The reason for this is that the function $A(x,y)=e^{x^2y}$ satisfies the symmetry $A(-x,y)=A(x,y)$ which induces additional cancellations of asymptotically leading terms.
\end{example}

In order to obtain precise asymptotic formulas in cases where the functions $A_m(x,y)$, as in the previous example, exhibit certain symmetries, we have to take into account these symmetries.
This can easily be accomplished by a small modification to our factorisation technique presented in Lemma~\ref{lem:det_factorisation}.
In fact, the only thing we have to do is to modify the representation~\eqref{eq:dominant_term_det+laurent}, the rest of our technique remains - mutatis mutandis - unchanged.

The following series of lemmas should illustrate these modifications to our factorisation method for some selected symmetry conditions, and should underline the general applicability of our factorisation method.

\begin{lemma}\label{lem:det_factorisation_taylor}
Let $A(x,y)$ be analytic for $(x,y)\in\mathcal R_1\times\mathcal R_2\subset\C^2$, where
\[ \mathcal R_j=\left\{ x\in\C\ :\ |x|<R_j^\ast \right\}, \qquad j=1,2, \]
for some $R_1^\ast,R_2^\ast>0$.
Furthermore, assume that $A(x,y)=-A(-x,y)$.

Then, the function
\[ \det_{1\le j,m\le k}\left( A(x_j,y_m) \right) \]
is analytic for $(x_1,\dots,x_k,y_1,\dots,y_k)\in\mathcal R_1^k\times\mathcal R_2^k$, and it satisfies
\begin{multline*}
\det_{1\le j,m\le k}\left( A(x_j,y_m) \right)
=
\left( \prod_{j=1}^kx_j\right)\left( \prod_{1\le j<m\le k}(x_m^2-x_j^2)(y_m-y_j) \right)\\
\times\det_{1\le j,m\le k}\left( \frac{1}{(2\pi i)^2}\int\limits_{\substack{|\xi|=R_1 \\ |\eta|=R_2}}\frac{A(\xi,\eta)d\xi d\eta}{\left(\prod\limits_{\ell=1}^j(\xi^2-x_\ell^2)\right)\left( \prod\limits_{\ell=1}^m(\eta-y_m) \right)}\right),
\end{multline*}
where $\max_{j}|x_j|<R_1<R_1^\ast$ and $\max_{j}|y_j|<R_2<R_2^\ast$.
\end{lemma}
\begin{proof}[Proof (sketch)]
By Cauchy's theorem, we have
\[ A(x,y)=\frac{1}{(2\pi i)^2}\int\limits_{\substack{|\xi|=R_1 \\ |\eta|=R_2}}\frac{A(\xi,\eta)d\xi d\eta}{(\xi-x)(\eta-y)}. \]
Now, our assumption implies $2A(x,y)=A(x,y)-A(-x,y)$, which, together with Cauchy's theorem, yields
\[ A(x,y)=\frac{x}{(2\pi i)^2}\int\limits_{\substack{|\xi|=R_1 \\ |\eta|=R_2}}\frac{A(\xi,\eta)d\xi d\eta}{(\xi^2-x^2)(\eta-y)}. \]
Consequently we should replace Equation~\eqref{eq:dominant_term_det+laurent} in the proof of Lemma~\ref{lem:det_factorisation} with
\[
\det_{1\le j,m\le k }\left( A(x_j,y_m) \right)
=
\left( \prod\limits_{j=1}^kx_j \right)\det_{1\le j,m\le k}\left( \frac{1}{(2\pi i)^2}\int\limits_{\substack{|\xi|=R_1 \\ |\eta|=R_2}}\frac{A(\xi,\eta)d\xi d\eta}{(\xi^2-x_j^2)(\eta-y_m)} \right).
\]
Now we apply the same sequence of row operations as in the proof of Lemma~\ref{lem:det_factorisation} to the determinant on the right hand side above.
After each of these row operations, we can take a factor of the form $(x_\ell^2-x_j^2)$, $\ell<j$, out of the determinant.

Finally, we transpose the resulting determinant (which does not change its value) and apply Lemma~\ref{lem:det_factorisation} with respect to $y_1,\dots,y_k$.
This proves the lemma.
%
\end{proof}

\begin{lemma}\label{lem:det_factorisation_1/symmetry}
	Let $A_m(x)$, $1\le m\le k$, be analytic and one-valued in the domain $\mathcal R=\left\{ x\in\C\ :\ \frac{1}{R^\ast}< |x|<R^\ast \right\}$ for some $R^\ast>1$.
Furthermore, assume that $A_m(x)=A_m(1/x)$, $m=1,\dots,k$.

Then, the function
\[ \det_{1\le j,m\le k}\left( A_m(x_j) \right) \]
is analytic for $(x_1,\dots,x_k)\in\mathcal R^k$, and it satisfies
\begin{multline*}
\det_{1\le j,m\le k}\left( A_m(x_j) \right)
=
\left( \prod_{j=1}^kx_j \right)^{-k}
\left( \prod_{1\le j<m\le k}(x_j-x_m)(1-x_jx_m) \right)
\left( \prod_{j=1}^k(x_j^2-1) \right)\\ \times
\det_{1\le j,m\le k}\left( \frac{1}{4\pi i}\int\limits_{|\xi|=R}\frac{A_m(\xi)\xi^{j-1}d\xi}{\prod\limits_{\ell=1}^j(\xi-x_\ell)\left( \xi-\frac{1}{x_\ell} \right)} - \frac{1}{4\pi i}\int\limits_{|\xi|=\frac{1}{R}}\frac{A_m(\xi)\xi^{j-1}d\xi}{\prod\limits_{\ell=1}^j(\xi-x_\ell)\left( \xi-\frac{1}{x_\ell} \right)} \right),
\end{multline*}
where $\frac{1}{R^\ast}<\frac{1}{R}<\min_{j}|x_j|\le\max_{j}|x_j|<R<R^\ast$.
\end{lemma}
\begin{proof}[Proof (sketch)]
Laurent's theorem together with our assumption $2A_m(x)=A_m(x)-A_m(1/x)$ implies
\[
A_m(x) = \frac{1}{4\pi i}\left( x-\frac{1}{x} \right)\left(\int\limits_{|\xi|=R}\frac{A_m(\xi)}{(\xi-x)\left( \xi-\frac{1}{x} \right)}d\xi-\int\limits_{|\xi|=\frac{1}{R}}\frac{A_m(\xi)}{(\xi-x)\left( \xi-\frac{1}{x} \right)}d\xi\right).
\]
Consequently, we should replace Equation~\eqref{eq:dominant_term_det+laurent} in the proof of Lemma~\ref{lem:det_factorisation} with
\begin{multline*}
\det_{1\le j,m\le k }\left( A_m(x_j) \right)
=
\left(\prod_{j=1}^k\left( x_j-\frac{1}{x_j} \right)\right) \\
\times\det_{1\le j,m\le k}\left(\frac{1}{4\pi i}\left(\int\limits_{|\xi|=R}\frac{A_m(\xi)d\xi}{(\xi-x_j)\left( \xi-\frac{1}{x_j} \right)}-\int\limits_{|\xi|=\frac{1}{R}}\frac{A_m(\xi)d\xi}{(\xi-x_j)\left( \xi-\frac{1}{x_j} \right)}\right)\right).
\end{multline*}
Now, short computations show that for $\frac{1}{R}\le\rho\le R$ we have
\begin{multline*}
\int\limits_{|\xi|=\rho}\frac{A_m(\xi)d\xi}{(\xi-x_\ell)\left( \xi-\frac{1}{x_\ell} \right)}
-\int\limits_{|\xi|=\rho}\frac{A_m(\xi)d\xi}{(\xi-x_j)\left( \xi-\frac{1}{x_j} \right)} \\
=\frac{1}{x_\ell x_j}(x_j-x_\ell)(1-x_j x_\ell)\int\limits_{|\xi|=\rho}\frac{A_m(\xi)\xi d\xi}{(\xi-x_\ell)\left( \xi-\frac{1}{x_\ell} \right)(\xi-x_j)\left( \xi-\frac{1}{x_j} \right)}.
\end{multline*}
We can now apply the same series of row operations as in the proof of Lemma~\ref{lem:det_factorisation}.
The only difference here is that after each row operation we take a factor of the form $x_\ell^{-1}x_j^{-1}(x_j-x_\ell)(1-x_jx_\ell)$, $j<\ell$, out of the determinant.
\end{proof}

The last class of determinants which is of importance in this manuscript and is not included in the preceding lemmas is considered in the next lemma.
\begin{lemma}\label{lem:det_factorisation_2}
	Let $A_m(x,y)$ and $B_m(x,y)$ be analytic in some non-empty domain $\mathcal R_1\times\mathcal R_2\subseteq \C^2$, where
\[ \mathcal R_j = \left\{ z\in\C\ :\ |z|<R_j^* \right\},\qquad j=1,2. \]
	Then, the function
	\[ \det_{1\le j,m\le k}\left( \begin{array}{l@{\hspace*{1cm}}l} A_m(x_j,y_m) & \textrm{for $j\le a$} \\ B_m(x_j,y_m) & \textrm{for $j>a$} \end{array} \right) \]
	is analytic for $(x_1,\cdots,x_k,y_1,\cdots,y_k)\in \mathcal R_1^k\times \mathcal R_2^k$, and satisfies
	\begin{multline*}
	\det_{1\le j,m\le k}\left( \begin{array}{l@{\hspace*{1cm}}l} A_m(x_j,y_m) & \textrm{for $j\le a$} \\ B_m(x_j,y_m) & \textrm{for $j>a$} \end{array} \right)=
		\left( \prod_{1\le j<m\le a}(x_m-x_j) \right)\left( \prod_{a<j<m\le k}(x_m-x_j) \right) \\
	\times\det_{1\le j,m\le k}\left( \begin{array}{l@{\hspace*{1cm}}l} \frac{1}{2\pi i}\int\limits_{|\xi|=R}\frac{A_m(\xi,y_m) d\xi}{\prod_{\ell=1}^j(\xi-x_\ell)} & \textrm{for $j\le a$} \\ \frac{1}{2\pi i}\int\limits_{|\xi|=R}\frac{B_m(\xi,y_m) d\xi}{\prod_{\ell=a+1}^j(\xi-x_\ell)} & \textrm{for $j>a$} \end{array} \right),
	\end{multline*}
	where $\max_j|x_j|<R<R_1^*$.
\end{lemma}
\begin{proof}[Proof (Sketch)]
	The proof is almost the same as the proof of Lemma~\ref{lem:det_factorisation}.
	The only difference is that we know apply the row operations separately to the rows $1,2,\cdots,a$ and to the rows $a+1,a+2,\cdots,k$.
\end{proof}

The rest of this section is devoted to some particular results that can be obtained by the above described technique.
More precisely, we determine asymptotics for some determinants that will become important in subsequent sections.
As illustrated in Example~\ref{ex:factorisation_technique}, asymptotics for determinants can be determined as follows.
First, we factorise our determinants according to our technique.
At this point it is important to take into account all the symmetries satisfied by the entries $A(x_j,y_m)$ of the determinant.
Second, we apply the geometric series expansion. This gives us the coefficient of the asymptotically leading term as a determinant, the entries of which being certain coefficients of the functions $A(x_j,y_m)$.

\begin{lemma}\label{lem:det_exp_asymptotics}
We have the asymptotics
\begin{multline*}
\det_{1\le j,m\le k}\left( e^{-(x_j-y_m)^2}-e^{-(x_j+y_m)^2} \right)
=
\left( \prod_{j=1}^kx_jy_j \right)\left( \prod_{1\le j<m\le k}(x_m^2-x_j^2)(y_m^2-y_j^2) \right) \\
\times\frac{2^{k^2+k}}{\prod_{j=1}^k(2j-1)!}
\left( 1+O\left( \sum_{j=1}^k(|x_j|^2+|y_j|^2) \right) \right)
\end{multline*}
as $x_1,\dots,x_k,y_1,\dots,y_k\to 0$.
\end{lemma}
\begin{proof}
The function $A(x,y)=e^{-(x-y)^2}-e^{-(x+y)^2}$ satisfies the requirements of Lemma~\ref{lem:det_factorisation_taylor}.
Therefore, we have
\begin{multline*}
\det_{1\le j,m\le k}\left( e^{-(x_j-y_m)^2}-e^{-(x_j+y_m)^2} \right)
=
\left( \prod_{j=1}^kx_jy_j \right)\left( \prod_{1\le j<m\le k}(x_m^2-x_j^2)(y_m^2-y_j^2) \right)\\
\times\det_{1\le j,m\le k}\left( \frac{1}{(2\pi i)^2}\int\limits_{\substack{|\xi|=1 \\ |\eta|=1}}\frac{A(\xi,\eta)d\xi d\eta}{\left(\prod\limits_{\ell=1}^j(\xi^2-x_\ell^2)\right)\left( \prod\limits_{\ell=1}^m(\eta^2-y_\ell^2) \right)}\right)
\end{multline*}
for $\max_j|x_j|,\max_j|y_j|<1$.
Since
\[
\frac{1}{(2\pi i)^2}\int_{\substack{|\xi|=1 \\ |\eta|=1}}
\left( e^{-(\xi-\eta)^2}-e^{-(\xi+\eta)^2} \right)\frac{d\xi}{\xi^{2j}}\frac{d\eta}{\eta^{2m}}
=
\frac{2}{(j+m-1)!}\binom{2j+2m-2}{2j-1},
\]
we deduce with the help of the geometric series expansion 
\begin{multline*}
\frac{1}{(2\pi i)^2}\int_{\substack{|\xi|=1 \\ |\eta|=1}}\frac{A(\xi,\eta)d\xi d\eta}{\left(\prod\limits_{\ell=1}^j(\xi^2-x_\ell^2)\right)\left( \prod\limits_{\ell=1}^m(\eta^2-y_\ell^2) \right)} \\
=
\frac{2}{(j+m-1)!}\binom{2j+2m-2}{2j-1}
\left( 1+O \left( \sum_{j=1}^k(|x_j|^2+|y_j|^2) \right)\right).
\end{multline*}
Consequently, we have
\begin{multline*}
\det_{1\le j,m\le k}\left( \frac{1}{(2\pi i)^2}\int\limits_{\substack{|\xi|=1 \\ |\eta|=1}}\frac{A(\xi,\eta)d\xi d\eta}{\left(\prod\limits_{\ell=1}^j(\xi^2-x_\ell^2)\right)\left( \prod\limits_{\ell=1}^m(\eta^2-y_\ell^2) \right)}\right)\\
=
\det_{1\le j,m\le k}\left( \frac{2}{(j+m-1)!}\binom{2j+2m-2}{2j-1} \right)
\left( 1+O\left( \sum_{j=1}^k(|x_j|^2+|y_j|^2) \right) \right).
\end{multline*}
The determinant on the right hand side can be evaluated into a closed form expression by taking some factors and applying \cite[Lemma 3]{MR1701596}, which gives us
\begin{align*}
	\det_{1\le j,m\le k}\left(\frac{2}{(j+m-1)!}\binom{2j+2m-2}{2j-1}\right)
	&= \frac{2^{k^2+k}}{\prod_{j=1}^k(2j-1)!},
\end{align*}
and completes the proof of the lemma.
\end{proof}

\begin{lemma}\label{lem:det_asymptotics_A(xy)}
	Let $A(x)=\sum_{\ell\ge 0}\alpha_\ell x^\ell$ be analytic for $|x|<R^*$, $R^*>0$.
Then, for any $a\in\Z$ and all $u_1,\dots,u_k\in\C$ the function
\[ \det_{1\le j,m\le k}\left( \begin{array}{c@{\hspace*{1cm}}l} A(u_mx_j) & \textrm{if $j\le a$} \\ (-1)^mA(u_mx_j) & \textrm{if $j>a$} \end{array} \right) \]
	satisfies the asymptotics 
\begin{multline*}
	\det_{1\le j,m\le k}\left( \begin{array}{c@{\hspace*{1cm}}l} A(u_mx_j) & \textrm{if $j\le a$} \\ (-1)^mA(u_mx_j) & \textrm{if $j>a$} \end{array} \right) = 
	2^{k-a}\left( \prod_{1\le j<m\le a}(x_m-x_j) \right)\left( \prod_{a<j<m\le k}(x_m-x_j) \right) \\
	\times \left( \left( \prod_{j=1}^a\alpha_{j-1} \right)\left( \prod_{j=1}^{k-a}\alpha_{j-1} \right)
	C(u_1,\dots,u_m)
	+O\left( \max_j|x_j| \right) \right)
\end{multline*}
as $x_1,\dots,x_k\to 0$, where
\[
C(u_1,\dots,u_m)=\det_{1\le j,m\le k}\left( \begin{array}{c@{\hspace*{1cm}}l} u_m^{j-1} & \textrm{if $j\le a$} \\ (-1)^mu_m^{j-a-1} & \textrm{if $j>a$} \end{array} \right).
\]
\end{lemma}
\begin{proof}
%
	An application of Lemma~\ref{lem:det_factorisation_2} (with $A_m(x)=A(x)$ and $B_m(x)=(-1)^mA(x)$) shows that
	\begin{multline*}
	\det_{1\le j,m\le k}\left( \begin{array}{c@{\hspace*{1cm}}l} A(u_mx_j) & \textrm{if $j\le a$} \\ (-1)^mA(u_mx_j) & \textrm{if $j>a$} \end{array} \right)
	=
	\left( \prod_{1\le j<m\le a}(x_m-x_j) \right)\left( \prod_{a< j<m\le k}(x_m-x_j) \right) \\
	\times\det_{1\le j,m\le k}\left( \begin{array}{c@{\hspace*{1cm}}l} \frac{1}{2\pi i}\int\limits_{|\xi|=R}\frac{A(u_m\xi) d\xi}{\prod_{\ell=1}^j(\xi-x_\ell)} & \textrm{if $j\le a$} \\ \frac{(-1)^m}{2\pi i}\int\limits_{|\xi|=R}\frac{A(u_m\xi) d\xi}{\prod_{\ell=a+1}^j(\xi-x_\ell)} & \textrm{if $j> a$} \end{array} \right),
	\end{multline*}
	where $0<R<R^*$.

	Now, the geometric series expansion followed by interchanging integration and summation further gives us
	\[ \frac{1}{2\pi i}\int\limits_{|\xi|=R}\frac{A(u_m\xi) d\xi}{\prod_{\ell=1}^j(\xi-x_\ell)} = \alpha_{j-1}u_m^{j-1}+O\left( \max_{j}|x_j| \right), \]
	from which we deduce that
	\begin{multline*}
\det_{1\le j,m\le k}\left( \begin{array}{c@{\hspace*{1cm}}l} \frac{1}{2\pi i}\int\limits_{|\xi|=R}\frac{A(u_m\xi) d\xi}{\prod_{\ell=1}^j(\xi-x_\ell)} & \textrm{if $j\le a$} \\ \frac{(-1)^m}{2\pi i}\int\limits_{|\xi|=R}\frac{A(u_m\xi) d\xi}{\prod_{\ell=a+1}^j(\xi-x_\ell)} & \textrm{if $j> a$} \end{array} \right) \\
	=
	\det_{1\le j,m\le k}\left( \begin{array}{c@{\hspace*{1cm}}l} \alpha_{j-1}u_m^{j-1} & \textrm{if $j\le a$} \\ (-1)^m\alpha_{j-a-1}u_m^{j-a-1} & \textrm{if $j>a$} \end{array} \right)
	+ O\left( \max_j|x_j| \right)
	\end{multline*}
	as $x_1,\dots,x_k\to 0$.
	This proves the claim.
\end{proof}
%
%

In general, the quantity $C(u_1,\dots,u_k)$ cannot be evaluated into a nice closed form expression.
Under some restrictions on the parameters $u_1,\dots,u_k$ however, we can at least prove that $C(u_1,\dots,u_k)$ is non zero.
This is of particular importance, if we want to compute asymptotics for rational expressions where the denominator involves determinants of the type considered in the last lemma.
(This is exactly the situation we have to deal with later on, in our analysis of walks with a free end point.)
\begin{lemma}
	If $0<u_1<\dots< u_k$ then
	\[
	\det_{1\le j,m\le k}\left( \begin{array} {c@{\hspace*{1cm}}l} (-1)^mu_m^{j-1} & \textrm{if $j\le a$} \\ u_m^{j-a-1} & \textrm{if $j>a$} \end{array} \right) \neq 0.
	\]
	\label{lem:non_zero_det_vandermondelike}
\end{lemma}
\begin{proof}
	For $a=0$, the claim is true as in that case the determinant is the Vandermonde determinant
	\[
	\Delta(u_\ell\ :\ 1\le\ell\le k) = \det_{1\le j,m\le k}\left( u_m^{j-1} \right)=
	\prod_{1\le j<m\le k}(u_m-u_j) > 0.
	\]

	For $a>0$, we consider the Laplace expansion of the determinant with respect to the first $a$ rows, viz.
	\begin{multline*}
	\sum_{f}\left(\prod_{j=1}^a(-1)^{j+f(j)}(-1)^{f(j)}u_{f(j)}^{j-1}\right)\Delta(u_\ell\ :\ \textrm{$\ell$ not in range of $f$}) \\
	= (-1)^{\binom{a+1}{2}}\sum_{f}\left( \prod_{j=1}^a u_{f(j)}^{j-1} \right)\Delta(u_\ell\ :\ \textrm{$\ell$ not in range of $f$}),
	\end{multline*}
	where the sum runs over all injective functions $f:\left\{ 1,2,\dots,a \right\}\to\left\{ 1,2,\dots,k \right\}$.
	Now, the claim follows upon noting that all addends of the sum on the right hand side of the equation are positive.
\end{proof}

\begin{lemma}\label{lem:det_sin_asymptotics}
For all $u_1,\dots,u_k\in\C$ we have the asymptotics 
\begin{multline*}
\det\limits_{1\le j,m\le k}\left( \sin(u_m\varphi_j) \right)
=
\left( \prod_{j=1}^k\varphi_j \right)\left( \prod_{1\le j<m\le k}(\varphi_m^2-\varphi_j^2) \right)
\left( \prod_{j=1}^k\frac{(-1)^j}{(2j-1)!} \right) \\
\times\left( \left( \prod_{j=1}^ku_j \right)\left( \prod_{1\le j<m\le k}(u_m^2-u_j^2) \right)\left(1-\frac{1}{2k(2k+1)}\left( \sum_{j=1}^k\varphi_j^2 \right)\right)+O\left( \max_j|\varphi_j|^4 \right) \right)
\end{multline*}
as $(\varphi_1,\dots,\varphi_k)\to(0,\dots,0)$.
\end{lemma}
\begin{proof}
An application of Lemma~\ref{lem:det_factorisation_taylor} shows that
\begin{multline*}
\det\limits_{1\le j,m\le k}\left( \sin(u_m\varphi_j) \right)
=
\left( \prod_{j=1}^k\varphi_j \right)\left( \prod_{1\le j<m\le k}(\varphi_m^2-\varphi_j^2) \right) \\
\times\det_{1\le j,m\le k}\left( \frac{1}{2\pi i}\int\limits_{|\eta|=1}\frac{\sin(u_m\eta) d\eta}{\left( \prod\limits_{\ell=1}^j\left(\eta^2-\varphi_\ell^2\right) \right)}\right).
\end{multline*}

Since we may assume that $\max_j|\varphi_j|<1$, we deduce by the geometric series expansion that
\begin{align*}
\frac{1}{2\pi i}\int\limits_{|\eta|=1}\frac{\sin(u_m\eta) d\eta}{\left( \prod\limits_{\ell=1}^j\left(\eta^2-\varphi_\ell^2\right) \right)}
&= \frac{1}{2\pi i}\int\limits_{|\eta|=1}\left( \frac{1}{\eta^{2j}}+\frac{1}{\eta^{2j+2}}\sum_{\ell=1}^k\varphi_\ell^2 \right)\sin(u_m\eta)d\eta
+O\left( \max_j|\varphi_j|^4 \right) \\
&= \frac{(-1)^{j-1}u_m^{2j-1}}{(2j-1)!}+\frac{(-1)^ju_m^{2j+1}}{(2j+1)!}\left( \sum_{\ell=1}^k\varphi_\ell^2 \right)+O\left( \max_j|\varphi_j|^4 \right)
\end{align*}
as $(\varphi_1,\dots,\varphi_k)\to(0,\dots,0)$.
This further shows that (after some simplifications)
\begin{multline*}
\det_{1\le j,m\le k}\left( \frac{1}{2\pi i}\int\limits_{|\eta|=1}\frac{\sin(u_m\eta) d\eta}{\left( \prod\limits_{\ell=1}^j\left(\eta^2-\varphi_\ell^2\right) \right)}\right)
=
(-1)^{\binom{k}{2}}\left(\prod_{\ell=1}^k\frac{u_\ell}{(2\ell-1)!}\right)\\
\times\det\limits_{1\le j,m\le k}\left( u_m^{2(j-1)}-\frac{u_m^{2j}}{2j(2j+1)}\sum_{\ell=1}^j\varphi_\ell^2 \right)+O\left( \max_j|\varphi_j|^4 \right)
\end{multline*}
as $(\varphi_1,\dots,\varphi_k)\to(0,\dots,0)$.
Let us now have a closer look at the expression
\[
P(u_1,\dots,u_k) = \det\limits_{1\le j,m\le k}\left( u_m^{2(j-1)}-\frac{u_m^{2j}}{2j(2j+1)}\sum_{\ell=1}^j\varphi_\ell^2 \right).
\]
Clearly, $P(u_1,\dots,u_k)$ is a polynomial in the variables $u_1,\dots,u_k$.
Furthermore, it is seen that $P(u_1,\dots,u_k)=0$ whenever $u_\ell=\pm u_m$ for some $\ell\neq m$, since in this case the $\ell$-th and the $m$-th column of the determinant on the right hand side are equal, and therefore, the determinant is equal to zero.
Consequently, we know that
\[
P(u_1,\dots,u_k)=\left( \prod_{1\le j<m\le k}(u_m^2-u_j^2) \right)\left( C+\sum_{\ell=1}^k\lambda_\ell u_\ell^2+O\left( \max_j|\varphi_j|^4 \right) \right)
\]
for some unknown coefficients $C,\lambda_1,\dots,\lambda_k$. At this point it should be noted that the $O$-term is in fact a polynomial in $u_1,\dots,u_k$ consisting of monomials of total degree $\ge 4$ only.
The unknown coefficients can now be determined by comparing certain monomials in the expression on the right hand side above and in the determinantal definition of $P(u_1,\dots,u_k)$.
In particular, comparing the coefficients of $\prod_{\ell=1}^ku_\ell^{2(\ell-1)}$ gives
\[ 
C = \det_{1\le j,m\le k}\left( \begin{array}{l@{\hspace*{0.5cm}}l} 1 & \textrm{if $j=m$} \\ -\frac{1}{2j(2j+1)}\sum_{\ell=1}^j\varphi_\ell^2 & \textrm{if $j=m-1$} \\ 0 & \textrm{else}\end{array} \right) = 1.
\]
On the other hand, comparing the coefficients of the monomials $u_m^2\prod_{\ell=1}^k u_{\ell}^{2(\ell-1)}$, $1\le m\le k$, we obtain the equations
\[
\begin{array}{l@{\hspace*{1cm}}l}
	\lambda_{m+1}-\lambda_m =	0 & \textrm{for $1\le m< k$,} \\
	\lambda_k =	\det\limits_{1\le j,m\le k}\left( \begin{array}{l@{\hspace*{0.5cm}}l} 1 & \textrm{if $j=m$ and $m<k$} \\ -\frac{1}{2j(2j+1)}\sum_{\ell=1}^j\varphi_\ell^2 & \textrm{if $j=m-1$ and $m<k$} \\ -\frac{1}{2k(2k+1)}\sum_{\ell=1}^k\varphi_\ell^2 & \textrm{if $j=m=k$} \\ 0 & \textrm{else} \end{array}\right) & \textrm{for $m=k$.}
\end{array}
\]
Here, the first set of equations follows by noting that the $m$-th and $(m+1)$-st column of the determinant are equal (which implies that the determinant is equal to zero).
From the recursion above we deduce that
\[
\lambda_1=\dots=\lambda_k=-\frac{1}{2k(2k+1)}\sum_{\ell=1}^k\varphi_\ell^2.
\]
This completes the proof of the lemma.
\end{proof}

\section{Walks with a fixed end point}\label{sec:fixedendpoint}

In this section, we are going to derive 
asymptotics for $P_n^+(\vec u\to\vec v)$ as $n$ tends to infinity (see Theorem~\ref{thm:asymptotic_u->v} below).
The asymptotics are derived by applying saddle point techniques to the integral representation~\eqref{eq:paths_u->v_coefficient_repr} together with the techniques developed in Section~\ref{sec:determinants_and_asymptotics}.

\begin{theorem}\label{thm:asymptotic_u->v}
	Let $\mathcal S$ be a composite step set over the atomic step set $\mathcal A$, and let $w:\mathcal S\to\R_+$ be a weight function.
	By $\mathcal L$ we denote the $\Z$-lattice spanned by $\mathcal A$.
	The composite step generating function associated with $\mathcal S$ is denoted by $S(z_1,\dots,z_k)$.
	Finally, let $\mathcal M\subseteq\left\{ 0,\pi \right\}^k$ denote the set of points such that the function $(\varphi_1,\dots,\varphi_k)\mapsto |S(e^{i\varphi_1},\dots,e^{i\varphi_k})|$ attains a maximum value, and let $|\mathcal M|$ denote the cardinality of the set $\mathcal M$.
	
	If $\mathcal A$, $\mathcal S$ and $w$ satisfy Assumption~\ref{ass:step_sets} and $S(1,\dots,1)>0$, 
	then for any two points $\vec u,\vec v\in\mathcal W^0\cap\mathcal L$ we have the asymptotic formula
%
\begin{multline}
P_{n}^+(\vec u\to\vec v) =
|\mathcal M|
S(1,\dots,1)^n
\left( \frac{2}{\pi} \right)^{k/2}
\left( \frac{1}{n\Lambda} \right)^{k^2+k/2} \\
\times
\frac{\left( \prod\limits_{1\le j<m\le k}(u_m^2-u_j^2)(v_m^2-v_j^2) \right)\left( \prod\limits_{j=1}^{k}u_jv_j \right)}{\left(\prod_{j=1}^{k}(2j-1)!\right)}
\left( 1+\frac{1}{n\Lambda}+O(n^{-5/3}) \right)
\end{multline}
as $n\to\infty$ in the set $\left\{n\ :\ P_{n}^+(\vec u\to\vec v)>0\right\}$. Here, $\Lambda=\frac{S''(1,\dots,1)}{S(1,\dots,1)}$ and $S''(z_1,\dots,z_k)$ denotes the second derivative of
$S(z_1,\dots,z_k)$ with respect to any of the $z_j$.
\end{theorem}
\begin{proof}
	We start from the exact integral expression for $P_n^+(\vec u\to\vec v)$ as given by Corollary~\ref{cor:exact_integral_u->v_alternative2}, viz.
\begin{multline*}
	P_{n}^+(\vec u\to\vec v)=
	\frac{1}{k!}\left(\frac{i}{4\pi}\right)^k\\ \times\idotsint\limits_{|z_1|=\dots=|z_k|=\rho}\det_{1\le j,m\le k}\left( z_j^{u_m}-z_j^{-u_m} \right)\det_{1\le j,m\le k}\left( z_j^{v_m}-z_j^{-v_m} \right)S(z_1,\dots,z_k)^n\left( \prod_{j=1}^{k}\frac{dz_j}{z_j}\right), 
\end{multline*}

The substitution $z_j=e^{i\varphi_j}$ then transforms this expression into
\begin{multline}\label{eq:thm1_proof_exact_integral}
P_{n}^+(\vec u\to\vec v) =
\frac{1}{k!\pi^k} \\
\times \int\limits_{-\pi}^{\pi}\!\dots\!\int\limits_{-\pi}^{\pi}
\det_{1\le j,m\le k}\Big( \sin(u_m\varphi_j) \Big)
\det_{1\le j,m\le k}\Big( \sin(v_m\varphi_j) \Big)
S\left( e^{i\varphi_1},\dots,e^{i\varphi_k} \right)^n\left(\prod_{j=1}^{k}d\varphi_j\right),
\end{multline}
which we are going to asymptotically evaluated in the following.

Now, the structure of the integral suggests that asymptotics as $n\to\infty$ restricted to the set $\left\{ n\ :\ P_n^+(\vec u\to\vec v)>0 \right\}$ can be extracted by means of a saddle point approach.
This means that the asymptotically dominant part of the integrand should be captured by small neighbourhoods around the maxima of $|S(z_1,\dots,z_k)|$ on the torus $|z_1|=\dots=|z_k|=1$.
Recall that, according to Lemma~\ref{lem:step_gen_fun_maxima}, the set $\mathcal M$ of these maxima is always contained in the set $\left\{ 0,\pi \right\}^k$.


For notational convenience, we define the sets
\[
\mathcal U_{\varepsilon}(\hat\varphi)=\left\{\varphi\in\R^k\ :\ |\hat{\varphi}-\varphi|_{\infty}<\varepsilon  \right\},
\qquad\hat{\varphi}=(\hat\varphi_1,\dots,\hat\varphi_k)\in\mathcal M,
\]
where $\varepsilon>0$ and $|\cdot|_{\infty}$ denotes the maximum norm on $\R^k$.
And we also set
\[
D_{\sin}(\vec u,\vec\varphi) = \det_{1\le j,m\le k}\Big( \sin(u_m\varphi_j) \Big).
\]
Now, we claim that the dominant asymptotic term of $P_{n}^+(\vec u\to\vec v)$ is captured by
\begin{equation}\label{eq:asymptotics_u->v_dominant_part}
\left( \frac{1}{k!\pi^k} \right)
\sum_{\hat\varphi\in\mathcal M}\idotsint\limits_{\mathcal U_{\varepsilon}(\hat\varphi)}
D_{\sin}(\vec u,\vec\varphi)
D_{\sin}(\vec v,\vec\varphi)
S\left( e^{i\varphi_1},\dots,e^{i\varphi_k} \right)^n\left(\prod_{j=1}^{k}d\varphi_j\right),
\end{equation}
where we choose $\varepsilon=\varepsilon(n)=n^{-5/12}$.
%
This claim can indeed be proved by means of the saddle point method:
(1) Determine an asymptotically equivalent expression for \eqref{eq:asymptotics_u->v_dominant_part}
that is more convenient to work with; (2) Find a bound for the remaining part of the integral~\eqref{eq:thm1_proof_exact_integral}.

Let us start with task (2) and establish a bound for the integral
\begin{equation*}
\left( \frac{1}{k!\pi^k} \right)
\idotsint\limits_{[0,2\pi)^k\setminus\mathcal U_{\varepsilon}(\mathcal M)}
D_{\sin}(\vec u,\vec\varphi)
D_{\sin}(\vec v,\vec\varphi)
S\left( e^{i\varphi_1},\dots,e^{i\varphi_k} \right)^n\left(\prod_{j=1}^{k}d\varphi_j\right),
\end{equation*}
where $\mathcal U_{\varepsilon}(\mathcal M)=\bigcup_{\hat\varphi\in\mathcal M}\mathcal U_{\varepsilon}(\hat\varphi)$ and
$\varepsilon=\varepsilon(n)=n^{-5/12}$.
Since $\mathcal M$ is the set of maximal points of the function $(\varphi_1,\dots,\varphi_k)\mapsto |S(e^{i\varphi_1},\dots,e^{i\varphi_k})|$,
we see that (at least for $n$ large enough) the maximum of this function on the set $[0,2\pi]^k\setminus\mathcal U_{\varepsilon}(\mathcal M)$
is attained somewhere on the boundary of one of the sets $\mathcal U_{\varepsilon}(\hat\varphi)$, $\hat\varphi\in\mathcal M$. 
Let $\vec\psi\in[0,2\pi]^k\setminus\mathcal U_{\varepsilon}(\mathcal M)$ be one such maximal point.
Since the expansion of Lemma~\ref{lem:step_gen_fun_asymptotics} is also valid for $\vec\psi$, we immediately obtain the upper bound
\[
\left|S\left( e^{i\varphi_1},\dots,e^{i\varphi_k} \right)\right|
\le\left|S\left( e^{i\psi_1},\dots,e^{i\psi_k} \right)\right|
=S(1,\dots,1)^{n-C^*n^{1/6}+O\left( n^{-1/4} \right)}
\]
as $n\to\infty$ for some constant $C^*>0$. This gives us the bound
\begin{multline*}
\left( \frac{1}{k!\pi^k} \right)
\idotsint\limits_{[0,2\pi)^k\setminus\mathcal U_{\varepsilon}(\mathcal M)}
D_{\sin}(\vec u,\vec\varphi)
D_{\sin}(\vec v,\vec\varphi)
S\left( e^{i\varphi_1},\dots,e^{i\varphi_k} \right)^n\left(\prod_{j=1}^{k}d\varphi_j\right) \\
=O\left( S(1,\dots,1)^{n-Cn^{1/6}} \right)
\end{multline*}
for some $0<C<C^*$ as $n\to\infty$.
In the following we will see that this is exponentially small compared to the contribution of \eqref{eq:asymptotics_u->v_dominant_part}.

Let us now focus on task (1) of the saddle point method and find an asymptotically equivalent expression for \eqref{eq:asymptotics_u->v_dominant_part}.
As a first remark, we note that for any $\hat\varphi\in\mathcal M\subseteq\left\{ 0,\pi \right\}^k$ and any $n$ such that $P_n^+(\vec u\to\vec v)>0$ we have
\begin{multline*}
D_{\sin}(\vec u,\vec\varphi+\hat{\vec\varphi})D_{\sin}(\vec v,\vec\varphi+\hat{\vec\varphi})S\left( e^{i(\varphi_1+\hat\varphi_j)},\dots,e^{i(\varphi_k+\hat\varphi_k)} \right)^n \\
=D_{\sin}(\vec u,\vec\varphi)D_{\sin}(\vec v,\vec\varphi)S(e^{i\varphi_1},\dots,e^{i\varphi_k})^n.
\end{multline*}
The validity of the above equation follows from Lemma~\ref{lem:step_gen_fun_1} and Lemma~\ref{lem:step_gen_fun_det_identity} as well as a case-by-case analysis based on Remark~\ref{rem:situations_casebycase}.
We omit the details and only note that in case (1) of Remark~\ref{rem:situations_casebycase} we have $\sum_{j}(u_j+v_j)\equiv n\mod 2$ whereas in case (2) we have $u_1-v_1\equiv\dots\equiv u_k-v_k\equiv n\mod 2$.

These considerations show that \eqref{eq:asymptotics_u->v_dominant_part} is equal to
\begin{equation}
\frac{|\mathcal M|}{k!\pi^k}
\idotsint\limits_{\mathcal U_{\varepsilon}(\vec 0)}
D_{\sin}(\vec u,\vec\varphi)
D_{\sin}(\vec v,\vec\varphi)
S\left( e^{i\varphi_1},\dots,e^{i\varphi_k} \right)^n\left(\prod_{j=1}^{k}d\varphi_j\right).
\label{eq:asymptotics u->v dominant part}
\end{equation}
Asymptotics for this last integral can now be derived by replacing the factors of the integrand with the appropriate Taylor series approximations, which can be found in Lemma~\ref{lem:det_sin_asymptotics} and Lemma~\ref{lem:step_gen_fun_asymptotics}.
This shows that for $\varepsilon=n^{-5/12}$ the integral \eqref{eq:asymptotics_u->v_dominant_part} is asymptotically equal to
\begin{multline}
\frac{|\mathcal M|}{k!\pi^k}\frac{\left( \prod\limits_{j=1}^ku_jv_j \right)\left( \prod\limits_{1\le j<m\le k}(u_m^2-u_j^2)(v_m^2-v_j^2) \right)}{\left( \prod\limits_{j=1}^k(2j-1)! \right)^2}S(1,\dots,1)^n \\
\times\left(
\left\langle 1\right\rangle_\varepsilon -\left\langle \frac{1}{k(2k+1)}\sum_{j=1}^k\varphi_j^2 \right\rangle_\varepsilon
+\left\langle O\left( \max_j|\varphi_j|^4 \right) \right\rangle_\varepsilon
\right)
\label{eq:asymptotics_u->v_dominant_pre_selberg}
\end{multline}
as $n\to\infty$ in the set $\left\{ n\ :\ P_n^+(\vec u\to\vec v)>0 \right\}$, where
\begin{equation*}
\left\langle f(\vec\varphi)\right\rangle_\varepsilon = \int\limits_{-\varepsilon}^\varepsilon\!\cdots\!\int\limits_{-\varepsilon}^\varepsilon
f(\vec\varphi)
\left( \prod_{j=1}^k\varphi_j^2 \right)\left( \prod_{1\le j<m\le k}(\varphi_m^2-\varphi_j^2) \right)^2
\exp\left(-\frac{n\Lambda}{2}\sum_{j=1}^k \varphi_j^2\right)\prod_{j=1}^kd\varphi_j.
\end{equation*}
First, we note that since $\varepsilon=n^{-5/12}$ we have $\max_j|\varphi_j|^4\le n^{-5/3}$ and consequently
\[
\left\langle O\left( \max_j|\varphi_j|^4 \right) \right\rangle_\varepsilon
= O\left( n^{-5/3}\langle 1\rangle_\varepsilon \right),
\qquad n\to\infty.
\]
Now, since in the present cases the function $f(\vec\varphi)=f(\varphi_1,\dots,\varphi_k)$ is even with respect to all of its arguments we can fold the integral to $[0,\varepsilon]$ and make the substitution $n\Lambda \varphi^2/2\mapsto\varphi$ thus obtaining
\begin{multline*}
	\selberg{f(\varphi)}_{\varepsilon}=
	\left( \frac{2}{n\Lambda} \right)^{k^2+k/2} \\
	\times\int\limits_{0}^{\varepsilon^2n\Lambda/2}\!\cdots\!\int\limits_{0}^{\varepsilon^2n\Lambda/2}
	f(\varphi)\left( \prod_{j=1}^k\sqrt{\varphi_j} \right)\left( \prod_{1\le j<m\le k}(\varphi_m-\varphi_j) \right)^2e^{-\sum_{j=1}^k\varphi_j}\prod_{j=1}^k d\varphi.
\end{multline*}
Now, our previous choice $\varepsilon=n^{-5/12}$ ensures that $\varepsilon^2n\Lambda/2\to\infty$ as $n\to\infty$,
and since $f(\vec\varphi)$ is of sub exponential growth, we may replace the upper bounds of the integrals by $+\infty$ introducing an exponentially small error (since $\int_y^\infty e^{-x}dx=O(e^{-y})$), that is
\[
\selberg{f(\vec\varphi)}_\varepsilon =
\left( \frac{2}{n\Lambda} \right)^{k^2+k/2}
\selberg{f\left( \frac{\vec\varphi}{\sqrt{n\Lambda/2}} \right)}_L + O\left( e^{-n\eta} \right),
\qquad n\to\infty,
\]
for some $\eta>0$, where $\selberg{f(\vec\varphi)}_L$ denotes the Selberg integral with respect to the Laguerre weight (see Appendix~\ref{app:selberg}).
These considerations together with Lemma~\ref{prop:selberg_eval} and Lemma~\ref{prop:selberg_sum_squarded_average} give us for the last factor of \eqref{eq:asymptotics_u->v_dominant_pre_selberg} the asymptotically equivalent expression
\[
\left\langle 1\right\rangle_\varepsilon -\left\langle \frac{1}{k(2k+1)}\sum_{j=1}^k\varphi_j^2 \right\rangle_\varepsilon
+\left\langle O\left( \max_j|\varphi_j|^4 \right) \right\rangle_\varepsilon
=
\left( \frac{2}{n\Lambda} \right)^{k^2+k/2}
\selberg{1}_L\left(1+\frac{1}{n\Lambda}+O\left( n^{-5/3} \right)  \right)
\]
as $n\to\infty$.
As a final step we recall that
\[ \selberg{1}_L=\pi^{k/2}2^{-k^2}\prod_{j=1}^k(2j-1)! \]
by Proposition~\ref{prop:selberg_eval}, and insert the resulting expression into \eqref{eq:asymptotics_u->v_dominant_pre_selberg}.
This completes the proof of the theorem.
\end{proof}

\section{Walks with a free end point}\label{sec:freeendpoint}

In this section, we are interested in the generating function $P_n^+(\vec u)$ for walks starting in $\vec u$ consisting of $n$ steps that are confined to the region $\mathcal W^0$.
This quantity can be written as the sum
\[ P_n^+(\vec u) = \sum_{\vec v\in\mathcal W^0}P_n^+(\vec u\to\vec v), \]
where $P_n^+(\vec u\to\vec v)$ denotes the generating functions for walks from $\vec u$ to $\vec v$ consisting of $n$ steps that are confined to the region $\mathcal W^0$.
This sum is in fact a finite sum, because there is only a finite number of points in $\mathcal W^0$ that are reachable from $\vec u$ in $n$ steps.
In order to find a nice expression for $P_n^+(\vec u)$ that is amenable to asymptotic methods, we proceed as follows.
First, we substitute the integral expression from Lemma~\ref{lem:exact_integral_u->v} for $P_n^+(\vec u\to\vec v)$ in the sum above.
In a second step, we interchange summation and integration.
This yields a sum that can be evaluated with the help of a known identity relating Schur functions and odd orthogonal characters (see Lemma~\ref{lem:identity_schur<->orthchar} below).
The resulting expression can then be asymptotically evaluated by means of saddle point techniques and the techniques from Section~\ref{sec:determinants_and_asymptotics}.

\begin{lemma}[see, e.g., {Macdonald~\cite[I.5]{MR553598}}]\label{lem:identity_schur<->orthchar}
	For any integer $c>0$, we have the identity
	\begin{equation}\label{eq:identity_schur<->orthchar}
		\sum_{0\le \lambda_1\le\dots\le\lambda_k\le 2c}\frac{\det\limits_{1\le j,m\le k}\left( z_j^{\lambda_m+m-1} \right)}{\det\limits_{1\le j,m\le k}\left( z_j^{m-1} \right)}
		=
		\frac{\det\limits_{1\le j,m\le k}\left( z_j^{2c+m-1/2}-z_j^{-(m-1/2)} \right)}{\det\limits_{1\le j,m\le k}\left( z_j^{m-1/2}-z_j^{-(m-1/2)} \right)}.
\end{equation}
\end{lemma}

\begin{remark}
Equation~\eqref{eq:identity_schur<->orthchar} is well-known in representation theory as well as in the theory of Young tableaux, but is usually given in a different form, for which we first need some notation.

For $\nu=(\nu_1,\dots,\nu_k)$, $\nu_1\ge\dots\ge\nu_k\ge0$, define the \emph{Schur function $s_\nu(z_1,\dots,z_k)$} by
\[
s_\nu(z_1,\dots,z_k)=\frac{\det\limits_{1\le j,m\le k}\left( z_j^{\nu_{m}+k-m} \right)}{\det\limits_{1\le j,m\le k}\left( z_j^{k-m} \right)},
\]
and further define for any $k$-tuple $\mu=(\mu_1,\dots,\mu_k)$ of integers or half-integers the \emph{odd orthogonal character} $\mathrm{so}_\mu(z_1^{\pm},\dots,z_k^{\pm},1)$ by
\[
\mathrm{so}_\mu(z_1^{\pm},\dots,z_k^{\pm},1)
=
\frac{\det\limits_{1\le j,m\le k}\left( z_j^{\mu_{m}+k-m+1/2}-z_j^{-(\mu_{m}+k-m+1/2)}\right)}{\det\limits_{1\le j,m\le k}\left( z_j^{k-m+1/2}-z_j^{-(k-m+/2)} \right)}.
\]
For details on Schur functions and odd orthogonal characters, we refer the reader to \cite{MR1153249}.
Combinatorial interpretations of Schur functions and odd orthogonal characters can be found in \cite{MR553598} and \cite{MR1426737,MR1271242,MR1035496}, respectively.

With the above notation at hand, we may rewrite Equation~\eqref{eq:identity_schur<->orthchar} as
\[
\sum_{2c\ge\nu_1\ge\dots\ge\nu_1\ge 0}s_{(\nu_1,\dots,\nu_k)}(z_1,\dots,z_k)
=
\left( \prod_{j=1}^kz_j \right)^c
\mathrm{so}_{(c,\dots,c)}(z_1^{\pm},\dots,z_k^{\pm},1).
\]
Proofs for this identity have been given by, e.g.,
Gordon~\cite{MR709701}, Macdonald~\cite[I.5, Example 16]{MR553598} and Stembridge~\cite[Corollary 7.4(a)]{MR1069389}.
An elementary proof of Lemma~\ref{lem:identity_schur<->orthchar} based on induction has been given by Bressoud~\cite[Proof of Lemma 4.5]{MR1718370}.

For a much more detailed account on this identity, we refer to \cite[Proof of Theorem 2]{MR1801472}.
\end{remark}

Theorems~\ref{thm:exact_u->} and \ref{thm:asymptotic_u->} below also rely on two results which we are going to summarise in the following lemmas.

\begin{lemma}[{see Krattenthaler~\cite[Lemma~2]{MR1701596}}]\label{lem:det_evaluations_adv_det_cal}
	We have the determinant evaluations
	\begin{align*}
		\det_{1\le j,m\le k }\left( z_j^m-z_j^{-m} \right) &=
		\left( \prod_{j=1}^k z_j \right)^{-k}\left( \prod_{1\le j<m\le k}(z_j-z_m)(1-z_jz_m) \right)\left( \prod_{j=1}^k(z_j^2-1) \right)\\
		\det_{1\le j,m\le k }\left( z_j^{m-1/2}-z_j^{-(m-1/2)} \right) &=
		\left( \prod_{j=1}^k z_j \right)^{-k+1/2}\left( \prod_{1\le j<m\le k}(z_j-z_m)(1-z_jz_m) \right)\left( \prod_{j=1}^k(z_j-1) \right).
	\end{align*}
\end{lemma}
We want to point out that the above evaluations could have also been determined by our determinant factorisation method presented in Section~\ref{sec:determinants_and_asymptotics} (in the spirit of Lemma~\ref{lem:det_factorisation_1/symmetry}).

\begin{lemma}\label{lem:det_quotient_laurent_polynomial}
For any non-negative integers $u_1,\dots,u_m$, the function
\[
\frac{\det\limits_{1\le j,m\le k}\left( x_j^{u_m}-x_j^{-u_m} \right)}{\det\limits_{1\le j,m\le k}\left( x_j^{m}-x_j^{-m} \right)}
\]
is a Laurent polynomial in the complex variables $x_1,\dots,x_k$.
\end{lemma}
For the correctness of this lemma we only note that (apart from a pre-factor) the expression is the quotient of two polynomials, where the all the zeros of the denominator are also zeros of the numerator. 
We also note that the quantity considered in this last lemma is known in the literature as a \emph{symplectic character}.
For details on symplectic characters we refer to \cite{MR1153249}.

\begin{theorem}\label{thm:exact_u->}
	Let $\mathcal S$ be a composite step set over the atomic step set $\mathcal A$.
	By $\mathcal L$ we denote the $\Z$-lattice spanned by $\mathcal A$.
	The composite step generating function associated with $\mathcal S$ is denoted by $S(z_1,\dots,z_k)$.

	If $\mathcal A, \mathcal S$ satisfy Assumption~\ref{ass:step_sets}, then for any point $\vec u=(u_1,\dots,u_k)\in\mathcal W^0\cap\mathcal L$ we have the exact formula
\begin{multline}\label{eq:exact_u->}
	P_n^+(\vec u) = \frac{(2\pi)^{-k}}{k!}
	\idotsint\limits_{|z_1|=\dots=|z_k|=\rho}
	\det_{1\le j,m\le k}(z_j^m)\det_{1\le j,m\le k}(z_j^{-m})
	\frac{\det\limits_{1\le j,m\le k}\left( z_j^{u_m}-z_j^{-u_m} \right)}{\det\limits_{1\le j,m\le k}\left( z_j^{m}-z_j^{-m} \right)}\\
	\times S(z_1,\dots,z_k)^n\left( \prod_{j=1}^{k}\frac{(z_j+1)dz_j}{z_j}\right),
\end{multline}
where $\rho>0$.
\end{theorem}
\begin{proof}
	We start from the exact expression for $P_n^+(\vec u\to \vec v)$ as given by Corollary~\ref{cor:exact_integral_u->v_alternative}, viz.
\begin{multline*}
	P_{n}^+(\vec u\to\vec v)\\ =
	\frac{(-1)^k}{(2\pi i)^kk!}\idotsint\limits_{|z_1|=\dots=|z_k|=\rho}\det_{1\le j,m\le k}\left( z_j^{u_m}-z_j^{-u_m} \right)S(z_1,\dots,z_k)^n\det_{1\le j,m\le k}\left( z_j^{v_m} \right)\left( \prod_{j=1}^{k}\frac{dz_j}{z_j}\right),
\end{multline*}
where we choose $0<\rho<1$.
We want to sum this expression over all $\vec v\in\mathcal W^0$.
This will be accomplished in two steps. First, we sum the expression above over all $\vec v=(v_1,\dots,v_k)\in\mathcal W^0$ such that $v_k\le 2c+k$ for some fixed $c$. Second, we let $c$ tend to infinity.

We have
\begin{multline*}
	\sum_{0<v_1<\dots<v_k\le 2c+k}P_n^+(\vec u\to\vec v)
	=\frac{(-1)^k}{(2\pi i)^kk!}
	\idotsint\limits_{|z_1|=\dots=|z_k|=\rho}
	\det_{1\le j,m\le k}\left( z_j^{u_m}-z_j^{-u_m} \right)S(z_1,\dots,z_k)^n\\ \times\left(\sum_{0<v_1<\dots<v_k\le 2c+k}\det_{1\le j,m\le k}\left( z_j^{v_m} \right)\right)\left( \prod_{j=1}^{k}\frac{dz_j}{z_j}\right).
\end{multline*}
Setting $\lambda_m=v_m-m$ in Lemma~\ref{lem:identity_schur<->orthchar}, we obtain
\[
\sum_{0<v_1<\dots< v_k\le 2c+k}\det_{1\le j,m\le k}\left( z_j^{v_m} \right)=
\det_{1\le j,m\le k}\left( z_j^{m} \right)
\frac{\det\limits_{1\le j,m\le k}\left( z_j^{2c+m-1/2}-z_j^{-(m-1/2)} \right)}{\det\limits_{1\le j,m\le k}\left( z_j^{m-1/2}-z_j^{-(m-1/2)} \right)}.
\]
Now, since $|z_j|=\rho<1$, we can let $c$ tend to infinity, and obtain
\begin{align*}
\sum_{0<v_1<\dots< \lambda_k}\det_{1\le j,m\le k}\left( z_j^{v_m} \right)&=
\det_{1\le j,m\le k}\left( z_j^{m} \right)
\frac{\det\limits_{1\le j,m\le k}\left( -z_j^{-(m-1/2)} \right)}{\det\limits_{1\le j,m\le k}\left( z_j^{m-1/2}-z_j^{-(m-1/2)} \right)} \\
&= (-1)^k\left(\prod_{j=1}^kz_j\right)^{1/2}\frac{\det\limits_{1\le j,m\le k}\left( z_j^{m} \right)\det\limits_{1\le j,m\le k}\left( z_j^{-m} \right)}{\det\limits_{1\le j,m\le k}\left( z_j^{m-1/2}-z_j^{-(m-1/2)} \right)}.
\end{align*}

Finally, we deduce from Lemma~\ref{lem:det_evaluations_adv_det_cal} that
\[
\det\limits_{1\le j,m\le k}\left( z_j^{m-1/2}-z_j^{-(m-1/2)} \right)
=
\left( \prod_{j=1}^k \frac{\sqrt{z_j}}{z_j+1}\right)
\det\limits_{1\le j,m\le k}\left( z_j^{m}-z_j^{-m} \right),
\]
which proves Equation~\eqref{eq:exact_u->} for $0<\rho<1$.

By Lemma~\ref{lem:det_quotient_laurent_polynomial}, the factor  
\[
\frac{\det\limits_{1\le j,m\le k}\left( z_j^{u_m}-z_j^{-u_m} \right)}{\det\limits_{1\le j,m\le k}\left( z_j^{m}-z_j^{-m} \right)}
\]
is a Laurent polynomial.
Hence, by Cauchy's theorem, the value of the integral~\eqref{eq:exact_u->} for $1\le\rho<\infty$ is the same as for $0<\rho<1$.
This proves the theorem.
\end{proof}

\begin{theorem}\label{thm:asymptotic_u->}
	Let $\mathcal S$ be a composite step set over the atomic step set $\mathcal A$.
	By $\mathcal L$ we denote the $\Z$-lattice spanned by $\mathcal A$.
	The composite step generating function associated with $\mathcal S$ is denoted by $S(z_1,\dots,z_k)$.

	If $\mathcal A, \mathcal S$ satisfy Assumption~\ref{ass:step_sets} and $S(1,\dots,1)>0$, then we have for any point $\vec u=(u_1,\dots,u_k)\in\mathcal W^0\cap\mathcal L$ the asymptotic formula
\begin{multline}\label{eq:asymptotic_u->}
P_n^+(\vec u)=
S(1,\dots,1)^n\left( \frac{2}{\pi} \right)^{k/2}\left( \frac{S(1,\dots,1)}{nS''(1,\dots,1)} \right)^{k^2/2} \\
\times \left( \prod_{j=1}^k\frac{u_j(j-1)!}{(2j-1)!} \right)\left( \prod_{1\le j<m\le k}(u_m^2-u_j^2) \right)\left( 1+O\left( n^{-1} \right) \right)
\end{multline}
as $n\to\infty$.
Here, $S''(1,\dots,1)$ denotes the second derivative of $S(z_1,\dots,z_k)$ with respect to any of the $z_j$.
\end{theorem}
\begin{remark}
	We note that although Theorem~\ref{thm:asymptotic_u->} is weaker compared to Theorem~\ref{thm:asymptotic_u->v}, our techniques would also allow us to determine the second order term of the asymptotic expansion of $P_n^+(\vec u)$ as $n\to\infty$.
\end{remark}

\begin{remark}
For the special case $\mathcal S\cong \mathcal A$ (i.e., $\mathcal S$ and $\mathcal A$ are isomorphic), the order of the asymptotic growth of $P_n^+(\vec u)$ has already been determined by Grabiner~\cite[Theorem 1]{grabiner}.
There, Grabiner gives the asymptotic growth order of the number of walks with a free end point in a Weyl chamber for any of the classical Weyl groups as the number of steps tends to infinity, but his method does not allow to determine the coefficient of the asymptotically dominant term.
\end{remark}

As a direct consequence of the Vandermonde formula and Lemma~\ref{lem:det_evaluations_adv_det_cal} we obtain the following result.
\begin{lemma}\label{lem:free:2}
	We have
	\[
	\frac{\det\limits_{1\le j,m\le k}\left( z_j^m \right)\det\limits_{1\le j,m\le k}\left( z_j^{-m} \right)}{\det\limits_{1\le j,m\le k}\left( z_j^m-z_j^{-m} \right)}=\frac{1}{\prod_{j=1}^k\left( z_j-\frac{1}{z_j} \right)}\prod_{1\le j<m\le k}\frac{2-\frac{z_m}{z_j}-\frac{z_j}{z_m}}{z_m+\frac{1}{z_m}-z_j-\frac{1}{z_j}}.
	\]
\end{lemma}

\begin{proof}[Proof of {Theorem~\ref{thm:asymptotic_u->}}]
Setting $\rho=1$ in  \eqref{eq:exact_u->} and substituting $z_j=e^{i\varphi_j}$, $j=1,2,\dots,k$ we first obtain the exact expression
\begin{multline*}
P_n^+(\vec u)=
\frac{1}{(2\pi)^kk!}
\int\limits_{-\pi}^\pi\dots\int\limits_{-\pi}^\pi
\left( \prod_{1\le j<m\le k}\left|e^{i\varphi_m}-e^{i\varphi_j}\right|^2 \right)
\frac{\det\limits_{1\le j,m\le k}\left(\sin\left( u_m\varphi_j \right)\right)}{\det\limits_{1\le j,m\le k}\left(\sin\left( m\varphi_j \right)\right)} \\
\times S\left( e^{i\varphi_1},\dots,e^{i\varphi_k} \right)^n\prod_{j=1}^k\left(1+e^{i\varphi_j}\right)d\varphi_j.
\end{multline*}
Now, writing $(1+e^{i\varphi_j})=e^{i\varphi_j/2}\left( e^{i\varphi_j/2}+e^{-i\varphi_j/2} \right)$, we further obtain (after adding the resulting integral to the one obtained by substituting $\varphi_j\mapsto -\varphi_j$, $1\le j\le k$ and dividing by $2$) the representation
\begin{multline}
	P_n^+(\vec u)=
	\frac{1}{\pi^kk!}
\int\limits_{-\pi}^\pi\dots\int\limits_{-\pi}^\pi
\left( \prod_{1\le j<m\le k}\left|e^{i\varphi_m}-e^{i\varphi_j}\right|^2 \right)
\frac{\det\limits_{1\le j,m\le k}\left(\sin\left( u_m\varphi_j \right)\right)}{\det\limits_{1\le j,m\le k}\left(\sin\left( m\varphi_j \right)\right)} \\
\times S\left( e^{i\varphi_1},\dots,e^{i\varphi_k} \right)^n\cos\left( \sum_{j=1}^k\frac{\varphi_j}{2} \right)\prod_{j=1}^k\cos\left( \frac{\varphi_j}{2} \right)d\varphi_j.
\label{eq:exact_u->_saddlepoint}
\end{multline}
Again, the structure of the integrand suggests that asymptotics as $n\to\infty$ can be established by means of a saddle point approach.
More precisely, we suspect that the asymptotically dominant contribution to the overall asymptotics of \eqref{eq:exact_u->_saddlepoint} comes from small neighbourhoods around the maximal points of the function $(\varphi_1,\dots,\varphi_k)\mapsto|S(e^{i\varphi_1},\dots,e^{i\varphi_k})|$.

First, we recall that, according to Lemma~\ref{lem:step_gen_fun_maxima}, the set $\mathcal M$ of these maxima is contained in the set $\{0,\pi\}^k$.
As in the proof of Theorem~\ref{thm:asymptotic_u->v}, we define for notational convenience the sets
\[
\mathcal U_\varepsilon(\vec{\hat\varphi})=\left\{ \vec\varphi\in\R^k\ :\ |\vec{\hat\varphi}-\vec\varphi|_\infty<\varepsilon \right\},
\qquad \vec{\hat\varphi}=(\hat\varphi_1,\dots,\hat\varphi_k)\in\mathcal M,
\]
where $\varepsilon=n^{-5/12}$ and $|\cdot|_\infty$ denotes the maximum norm on $\R^k$.

The proof of this theorem will follow very much the lines of the proof of Theorem~\ref{thm:asymptotic_u->v}.
We will therefore stay rather brief and refer the reader to the proof of Theorem~\ref{thm:asymptotic_u->v} for details.
Let us now proceed with the application of the saddle point method:
(1) Determine an asymptotically equivalent expressions for the integrand of \eqref{eq:exact_u->_saddlepoint} valid in $\mathcal U_\varepsilon(\hat\varphi)$, $\hat\varphi\in\mathcal M$;
(2) Find a bound for the complementary part of the integral~\eqref{eq:exact_u->_saddlepoint}.

Let us start with Task (2).
The very same arguments as in the proof of Theorem~\ref{thm:asymptotic_u->v} show that
\begin{multline*}
	\frac{1}{\pi^kk!}
	\idotsint\limits_{ [-\pi,\pi]^k\setminus\mathcal U_\varepsilon(\mathcal M)}
\left( \prod_{1\le j<m\le k}\left|e^{i\varphi_m}-e^{i\varphi_j}\right|^2 \right)
\frac{\det\limits_{1\le j,m\le k}\left(\sin\left( u_m\varphi_j \right)\right)}{\det\limits_{1\le j,m\le k}\left(\sin\left( m\varphi_j \right)\right)} \\
\times S\left( e^{i\varphi_1},\dots,e^{i\varphi_k} \right)^n\cos\left( \sum_{j=1}^k\frac{\varphi_j}{2} \right)\prod_{j=1}^k\cos\left( \frac{\varphi_j}{2} \right)d\varphi_j \\
=O\left( S(1,\dots,1)^{n-Cn^{1/6}} \right)
\end{multline*}
for some $C>0$ as $n\to\infty$, where $\mathcal U_\varepsilon(\mathcal M)=\bigcup_{\vec{\hat\varphi}}\mathcal U_\varepsilon(\vec{\hat\varphi})$.
In the following we will see that this is exponentially small compared to the asymptotic behaviour of \eqref{eq:exact_u->_saddlepoint}.

We now show how to accomplish Task (1).
For this, we need to determine the Taylor series expansion of the integrand in Equation~\eqref{eq:exact_u->_saddlepoint} around the finitely many points $\hat\varphi\in\mathcal M$.
As the integrand is invariant under permutation of the integration variables $\varphi_1,\dots,\varphi_k$, we may therefore assume without the loss of generality  that $\hat\varphi\in\mathcal M$ is such that
\[ \pi=\hat\varphi_1=\dots=\hat\varphi_a>\hat\varphi_{a+1}=\dots=\hat\varphi_k=0\]
for some $0\le a\le k$ (with the obvious interpretation for $a=0$ and $a=k$).
Now, since $\left|e^{i\varphi_m}-e^{i\varphi_j}\right|^2=2-2\cos\left( \varphi_m-\varphi_j \right)$, we conclude that
\begin{multline*}
\prod_{1\le j<m\le k}\left|e^{i(\hat\varphi_m+\varphi_m)}-e^{i(\hat\varphi_j+\varphi_j)}\right|^2
=
2^{a(k-a)}
\left( \prod_{1\le j<m\le a}(\varphi_m-\varphi_j)^2 \right)
\left( \prod_{a< j<m\le k} (\varphi_m-\varphi_j)^2 \right) \\
\times\left( 1+O\left( \max_{j}|\varphi_j|^2 \right) \right)
\end{multline*}
as $(\varphi_1,\dots,\varphi_k)\to(0,\dots,0)$.
And since $\cos\left( \frac{\pi}{2}+\varphi \right)=-\sin(\varphi)$ we see that
\begin{multline*}
\cos\left( \sum_{j=1}^k\frac{\hat\varphi_j+\varphi_j}{2} \right)\prod_{j=1}^k\cos\left( \frac{\hat\varphi_j+\varphi_j}{2} \right) \\
=
(-1)^{a+\ceil{a/2}}4^{-\ceil{a/2}}(\varphi_1+\dots+\varphi_k)^{2\ceil{a/2}-a}\left( \prod\limits_{1\le j\le a}\varphi_j \right)
\left( 1+O\left( \max\limits_j|\varphi_j|^2 \right) \right)
\end{multline*}
as $(\varphi_1,\dots,\varphi_k)\to(0,\dots,0)$.
By Lemma~\ref{lem:step_gen_fun_det_identity}, we know that
\begin{multline*}
\det_{1\le j,m\le k}\Big( \sin\left( u_m\left( \hat\varphi_j+\varphi_j \right) \right) \Big)S\left( e^{i(\hat\varphi_1+\varphi_1)},\dots,e^{i(\hat\varphi_k+\varphi_k)} \right)^n\\ 
=
\left( \frac{S\left( e^{i\hat\varphi_1},\dots,e^{i\hat\varphi_k} \right)}{S(1,\dots,1)} \right)^n\det_{1\le j,m\le k}\Big( \sin\left( u_m\varphi_j\right) \Big)S\left( e^{i\varphi_1},\dots,e^{i\varphi_k} \right)^n
\end{multline*}
Asymptotics for the last two factors on the right hand side can be found in Lemma~\ref{lem:det_sin_asymptotics} and Lemma~\ref{lem:step_gen_fun_asymptotics} respectively.
For the second determinant we note that
\[
	\det_{1\le j,m\le k}\Big( \sin\left( m\left( \hat\varphi_j+\varphi_j \right) \right) \Big)
	=\det_{1\le j,m\le k}\left( \begin{array}{rl} (-1)^{m}\sin(m\varphi_j) & j\le a \\ \sin(m\varphi_j) & j>a \end{array} \right).
\]
Hence, we find after a short computation in the spirit of Lemma~\ref{lem:det_asymptotics_A(xy)} and Lemma~\ref{lem:det_sin_asymptotics} that
\begin{multline*}
\det_{1\le j,m\le k}\Big( \sin\left( m\left( \hat\varphi_j+\varphi_j \right) \right) \Big)=
\left( \prod_{j=1}^k\varphi_j \right)\left( \prod_{1\le j<m\le a}(\varphi_m^2-\varphi_j^2) \right)\left( \prod_{a<j<m\le k}(\varphi_m^2-\varphi_j^2) \right) \\
\times\left( \prod_{j=1}^a\frac{(-1)^j}{(2j-1)!} \right)\left( \prod_{j=a+1}^k\frac{(-1)^{j-a}}{(2j-2a-1)!} \right)\det_{1\le j,m\le k}\left( \begin{array}{rl}  (-1)^mm^{2j-1} & j\le a \\ m^{2(j-a)-1} & j>a \end{array} \right) \\
\times\left( 1+O\left( \max_{j}|\varphi_j|^2 \right) \right)
\end{multline*}
as $(\varphi_1,\dots,\varphi_k)\to(0,\dots,0)$.
Note that, by Lemma~\ref{lem:non_zero_det_vandermondelike}, the determinant on the right hand side above is non zero.

Hence, in a neighbourhood of the point $\hat\varphi\in\mathcal M$ with $\pi=\hat\varphi_1=\dots=\hat\varphi_a>\hat\varphi_{a+1}=\dots=\hat\varphi_k=0$, the integrand of Equation~\eqref{eq:exact_u->_saddlepoint} admits the asymptotic expansion
\begin{multline}
	\frac{4^{-\ceil{a/2}}}{\pi^kk!}
	(-1)^{a+\ceil{a/2}+\binom{k+1}{2}+\binom{a+1}{2}+\binom{k-a+1}{2}}\frac{\left( \prod_{j=1}^a(2j-1)! \right)\left( \prod_{j=a+1}^k(2j-2a-1)! \right)}{\prod_{j=1}^k(2j-1)!} \\
\times\frac{\left( \prod_{j=1}^ku_j \right)\left( \prod_{1\le j<m\le k}(u_m^2-u_j^2) \right)}{\det\limits_{1\le j,m\le k}\left( \begin{array}{rl}  (-1)^mm^{2j-1} & j\le a \\ m^{2(j-a)-1} & j>a \end{array} \right)}
S\left( e^{i\hat\varphi_1},\dots,e^{i\hat\varphi_k} \right)^n
\\
\times
\left( \prod\limits_{1\le j\le a}\varphi_j \right)(\varphi_1+\dots+\varphi_k)^{2\ceil{a/2}-a}
\left( \prod_{j=1}^a\prod_{m=a+1}^k\frac{\varphi_m+\varphi_j}{\varphi_m-\varphi_j} \right)
\\
\times\left( \prod_{1\le j<m\le k}(\varphi_m-\varphi_j)^2 \right)	
\exp\left( -n\Lambda\sum_{j=1}^k\frac{\varphi_j^2}{2} \right)
\left( 1+O\left( \max\limits_j|\varphi_j|^2 \right) \right)
\label{eq:asymptotics_u->dominant_part_integrand}
\end{multline}
as $(\varphi_1,\dots,\varphi_k)\to(0,\dots,0)$, where $\Lambda=\frac{S''(1,\dots,1)}{S(1,\dots,1)}>0$ and $S''(z_1,\dots,z_k)=\frac{\partial^2}{\partial z_1^2}S(z_1,\dots,z_k)$.

Asymptotics for the integral over $[-\varepsilon,\varepsilon]^k$, $\varepsilon=n^{-5/12}$, of the expression above can now be determined in a completely analogous fashion to the way we chose in the proof of Theorem~\ref{thm:asymptotic_u->v}:
First, we make the substitution $\varphi_j\mapsto\varphi_j\sqrt{n\Lambda}$, $j=1,\dots,k$.
The resulting integral is an integral over $[-\varepsilon\sqrt{n\Lambda},\varepsilon\sqrt{n\Lambda}]^k$.
Now, our previous choice $\varepsilon=n^{-5/12}$ ensures that $\varepsilon \sqrt{n}\to\infty$ as $n\to\infty$.
Hence, we may replace the range of integration by $(-\infty,\infty)^k$ introducing an exponentially small error (as $n\to\infty$) only, which is negligible in our considerations.
This shows that the integral of interest is asymptotically equal to
\begin{multline*}
	\frac{4^{-\ceil{a/2}}}{\pi^kk!}
	(-1)^{a+\ceil{a/2}+\binom{k+1}{2}+\binom{a+1}{2}+\binom{k-a+1}{2}}\frac{\left( \prod_{j=1}^a(2j-1)! \right)\left( \prod_{j=a+1}^k(2j-2a-1)! \right)}{\prod_{j=1}^k(2j-1)!} \\
\times\frac{\left( \prod_{j=1}^ku_j \right)\left( \prod_{1\le j<m\le k}(u_m^2-u_j^2) \right)}{\det\limits_{1\le j,m\le k}\left( \begin{array}{rl}  (-1)^mm^{2j-1} & j\le a \\ m^{2(j-a)-1} & j>a \end{array} \right)}
S\left( e^{i\hat\varphi_1},\dots,e^{i\hat\varphi_k} \right)^n
\left( \frac{1}{n\Lambda} \right)^{k^2/2+\ceil{a/2}}
\\
\selberg{(\varphi_1+\dots+\varphi_k)^{2\ceil{a/2}-a}\left( \prod\limits_{1\le j\le a}\varphi_j \right)
\left( \prod_{j=1}^a\prod_{m=a+1}^k\frac{\varphi_m+\varphi_j}{\varphi_m-\varphi_j} \right)}_H
\times
\left( 1+O\left( n^{-5/6} \right) \right)
\end{multline*}
as $n\to\infty$, where $\selberg{\cdot}_H$ denotes the Selberg-type integral with respect to the Hermite weight (see Appendix~\ref{app:selberg}) of the form
\[
\selberg{f(\vec\varphi)}_H = \int\limits_{-\infty}^\infty\!\cdots\!\int\limits_{-\infty}^\infty f(\vec\varphi)\left( \prod_{1\le j<m\le k}(\varphi_m-\varphi_j) \right)^2e^{-\sum_{j=1}^k\varphi_j^2/2} d\vec \varphi.
\]
It should be noted that the factor $(\varphi_1+\dots+\varphi_k)^{2\ceil{a/2}-a}\left( \prod\limits_{1\le j\le a}\varphi_j \right)$ entails, as we have seen above, the factor $n^{-\ceil{a/2}}$.
Consequently, the asymptotically dominant behaviour of \eqref{eq:exact_u->_saddlepoint} is completely captured by the point $(0,\dots,0)\in\mathcal M$, while the other maximal points contribute terms of the order $O\left( S(1,\dots,1)^nn^{k^2/2-1} \right)$ or lower only.
For $a=0$, i.e., the maximal point $(0,\dots,0)$, the asymptotics above simplifies to
\[
\frac{1}{\pi^kk!}\left( \prod_{j=1}^k\frac{u_j}{j} \right)\left( \prod_{1\le j<m\le k}\frac{u_m^2-u_j^2}{m^2-j^2} \right)
S(1,\dots,1)^n\left( \frac{1}{n\Lambda} \right)^{k^2/2}\selberg{1}_H\left( 1+O\left( n^{-5/6} \right) \right)
\]
as $n\to\infty$.
Now, by Proposition~\ref{prop:selberg_eval}, we have
\[\selberg{1}_H = (2\pi)^{k/2}\prod_{j=1}^kj!, \]
and easy calculations show that
\[ \left( \prod_{j=1}^kj \right)\left( \prod_{1\le j<m\le k}(m^2-j^2) \right) = \prod_{j=1}^k(2j-1)!. \]
This proves the theorem.

The reader may now object that in the statement of the theorem we claimed an error term of order $O\left( n^{-1} \right)$, while in the proof we obtained an error term of order $O\left( n^{-5/6} \right)$ only.
This is true, of course, but the reasons for this were merely didactical ones.
Indeed, a careful analysis of our proof reveals that we can improve the error term to $O\left( n^{-1} \right)$ as follows.
In our proof we simply replaced the factor $1+O\left( \max_j|\varphi_j|^2 \right)$ in \eqref{eq:asymptotics_u->dominant_part_integrand} with $1+O\left( n^{-5/6} \right)$, because $|\varphi_j|\le \varepsilon=n^{-5/12}$.
But a more careful expansion of our integrand shows that this error term can more precisely be described by
\[ 1+p(\varphi_1,\dots,\varphi_k)+O\left( \max_j|\varphi_j|^4 \right),\qquad n\to\infty, \]
where $p(\varphi_1,\dots,\varphi_k)$ is a homogeneous polynomial of degree $2$.
Now, proceeding as described above, we see that this polynomial (due to the substitution $\varphi_j\mapsto \varphi_j\sqrt{n\Lambda}$) yields the error term $O\left( n^{-1} \right)$.
This finally also settles the error term, and completes the proof of the theorem.

\end{proof}

\section{Applications}\label{sec:applications}

This section is entirely devoted to applications of Theorem~\ref{thm:asymptotic_u->v} and Theorem~\ref{thm:asymptotic_u->}.
Some results (or special cases thereof) presented in Subsection~\ref{sec:applications:lock_step_model} have already been derived earlier by other authors.
Some other results  in Subsection~\ref{sec:applications:randomturns_model} (in particular, Corollaries~\ref{cor:lockstep_fixed}, \ref{cor:lockstep_free}, \ref{cor:randomturns_fixed} and \ref{cor:randomturns_free}) seem, to the author's best knowledge, to be new.
Subsections~\ref{sec:applications:tangleddiags:isolated} and \ref{sec:applications:tangleddiags} contain precise asymptotics for which, up to the present time, only the order of growth was known.

\subsection{Lock step model of vicious walkers with wall restriction}
\label{sec:applications:lock_step_model}
In general, the vicious walkers model is concerned with $k$ random walkers on a $d$-dimensional lattice.
In the lock step model, at each time step all of the walkers move one step in any of the allowed directions,
such that at no time any two random walkers share the same lattice point.
This model was defined by Fisher~\cite{MR751710} as a model for wetting and melting processes.

In this subsection, we consider a \emph{two dimensional lock step model of vicious walkers with wall restriction}, which we briefly describe now.
The only allowed steps are $(1,1)$ and $(1,-1)$, and the lattice is the $\Z$-lattice spanned by these
two vectors.
Fix two vectors $\vec u,\vec v\in\Z^k$ such that $0< u_1<u_2<\dots<u_k$ and $u_i\equiv u_j\mod 2$ for $1\le i<j\le k$, and analogously for $\vec v$.
For $1\le j\le k$, the $j$-th walker starts at $(0,u_j-1)$ and, after $n$ steps, ends at the point $(n,v_j-1)$ in a way
such that at no time the walker moves below the horizontal axis (``the wall'') or shares a lattice point with another walker.

Certain configurations of the two dimensional vicious walkers model, such as \emph{watermelons} and \emph{stars} consisting of $k$ vicious walkers with or without the presence of an impenetrable walls, have been fully analysed by Guttmann et al.~\cite{MR1651492}
and Krattenthaler et al.~\cite{MR1801472,MR1964695}.
In their papers, they prove exact as well as asymptotic results for the total number of these configurations.

The results in this subsection include asymptotics for the total number of vicious walkers configurations with an arbitrary (but fixed) starting point having either an arbitrary (but fixed) end point or a free end point (see Corollary~\ref{cor:lockstep_fixed} and Corollary~\ref{cor:lockstep_free}, respectively).
Special cases of these asymptotics have been derived earlier by Krattenthaler et al.~\cite{MR1801472,MR1964695} and Rubey~\cite{rubey}.
For further links to the literature concerning this model, we refer to the references given in the papers
mentioned before.

%
The two dimensional lock step model of vicious walkers as described above can easily be reformulated as a model of lattice paths
in a Weyl chamber of type $B$ as follows: at each time, the positions of the walkers are encoded by a $k$-dimensional vector,
where the $j$-th coordinate records the current second coordinate (the \emph{height}) of the $j$-th walker. Clearly, if $(c_1,\dots,c_k)\in\Z^k$ is such a vector
encoding the heights of our walkers at a certain point in time, then we necessarily have $0\le c_1<c_2<\dots<c_k$ and
$c_i\equiv c_j\mod 2$ for $1\le i<j\le k$.
Hence, each realisation of the lock step model with $k$ vicious walkers, where the $j$-th walker starts at $(0,u_j-1)$ and
ends at $(n,v_j-1)$, naturally corresponds to a lattice path in
\[
\Big\{ (x_1,x_2,\dots,x_k)\in\Z^k\ :\ 0<x_1<\dots<x_k \textrm{ and $x_i\equiv x_j\!\mod 2$ for $1\le i<j\le k$}\Big\}
\]
that starts at $\vec u=(u_1,\dots,u_k)$ and ends at $\vec v=(v_1,\dots,v_k)$. (Note the shift by $+1$.)
For an illustration of this correspondence see Figure~\ref{fig:viciouswalkers_correspondance}.
The atomic step set is given by 
\[ \mathcal A=\left\{ \sum_{j=1}^{k}\varepsilon_j\vec e^{(j)}\ :\ \varepsilon_1,\dots,\varepsilon_k\in\left\{ -1,+1 \right\} \right\}, \]
and the composite step set $\mathcal S$ is set of all sequences of length one of elements in $\mathcal A$.
This means, that in the present case there is only a formal difference between the atomic steps and composite steps.
Both sets, $\mathcal A$ and $\mathcal S$ satisfy Assumption~\ref{ass:step_sets} (the conditions of Lemma~\ref{lem:reflection_principle}).
Consequently, asymptotics for this model can be obtained from Theorem~\ref{thm:asymptotic_u->v} and Theorem~\ref{thm:asymptotic_u->}.

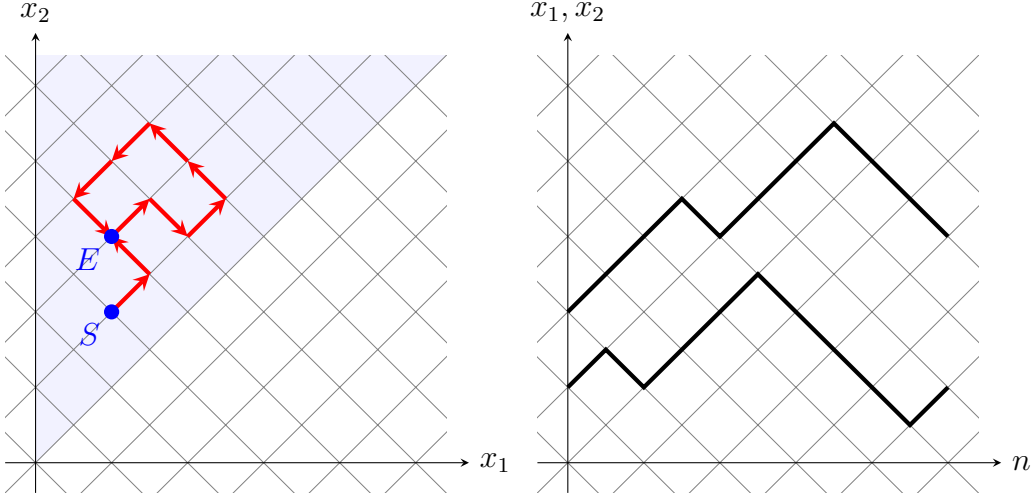
\begin{figure}
\begin{tikzpicture}[>=stealth]
	\scope
	\clip (-.4,-.4) rectangle (5.4,5.4);
	\fill[blue!5!white] (0,0) -- (10,10) -- (0,10) -- cycle; 
	\draw[step=1cm,gray,very thin,rotate=45,scale=0.70710678119] (-10,-10) grid (12,12);
	\endscope
	\draw[->,thin] (-.4,0) -- (5.7,0) node[anchor=west]  {$x_1$};
	\draw[->,thin] (0,-.4) -- (0,5.7) node[anchor=south]  {$x_2$};
	\scope[->,ultra thick,color=red]
	\draw (1,2) -- (1.5,2.5);
	\draw (1.5,2.5) -- (1,3);
	\draw (1,3) -- (1.5,3.5);
	\draw (1.5,3.5) -- (2,3);
	\draw (2,3) -- (2.5,3.5);
	\draw (2.5,3.5) -- (2,4);
	\draw (2,4) -- (1.5,4.5);
	\draw (1.5,4.5) -- (1,4);
	\draw (1,4) -- (.5,3.5);
	\draw (.5,3.5) -- (1,3);
	\endscope
	\fill[fill=blue] (1,2) node[anchor=north east,blue,thick] {$S$} circle (.1cm);
	\fill[fill=blue] (1,3) node[anchor=north east,blue,thick] {$E$} circle (.1cm);

	\scope[xshift=7cm]
	\scope
	\clip (-.4,-.4) rectangle (5.4,5.4);
	\draw[step=1cm,gray,very thin,rotate=45,scale=0.70710678119] (-10,-10) grid (12,12);
	\endscope
	\draw[->,thin] (-.4,0) -- (5.7,0) node[anchor=west]  {$n$};
	\draw[->,thin] (0,-.4) -- (0,5.7) node[anchor=south]  {$x_1, x_2$};
	\scope[ultra thick,color=black]
	\draw (0,1) -- (.5,1.5) -- (1,1) -- (1.5,1.5) -- (2,2) -- (2.5,2.5) -- (3,2) -- (3.5,1.5) -- (4,1) -- (4.5,.5) -- (5,1); 
	\draw (0,2) -- (.5,2.5) -- (1,3) -- (1.5,3.5) -- (2,3) -- (2.5,3.5) -- (3,4) -- (3.5,4.5) -- (4,4) -- (4.5,3.5) -- (5,3);
	\endscope
	\endscope
\end{tikzpicture}
\caption{
Illustration of the correspondence between walks in a Weyl chamber and the lock step model of vicious walkers.
On the left: a $10$-step walk from $S=(2,4)$ to $E=(2,6)$ restricted to the Weyl chamber $0<x_1<x_2$ (indicated by the shaded region).
On the right: the corresponding pair of non-intersecting lattice paths: the lower path from $(0,2)$ to $(10,2)$ keeps track of the horizontal coordinate (minus $1$) of the walk in the left hand side, while the upper path from $(0,3)$ to $(10,5)$ keeps track of the vertical coordinate.
}
\label{fig:viciouswalkers_correspondance}
\end{figure}

The composite step generating function associated with $\mathcal S$ is
\[ S(z_1,\dots,z_k)=\prod_{j=1}^{k}\left( z_j+\frac{1}{z_j} \right), \]
and the set $\mathcal M\subseteq\left\{ 0,\pi \right\}^k$ of points maximising the function $(\varphi_1,\dots,\varphi_k)\mapsto|S(e^{i\varphi_1},\dots,e^{i\varphi_k})|$ is given by $\mathcal M=\left\{ 0,\pi \right\}^k$.
Hence, we have $|\mathcal M|=2^k$,
and after short calculations we find $S(1,\dots,1)=S''(1,\dots,1)=2^k$.
As a consequence of Theorem~\ref{thm:asymptotic_u->v}, we obtain the following result.
\begin{corollary}\label{cor:lockstep_fixed}
The number of vicious walkers of length $n$ with $k$ walkers
that start at $(0,u_1-1),\dots,(0,u_k-1)$ and end at $(n,v_1-1),\dots,(n,v_k-1)$ (we assume that $u_1+v_1\equiv n\mod 2$) is asymptotically
equal to
\[
2^{nk+3k/2}\pi^{-k/2}n^{-k^2-k/2}
\frac{\left( \prod\limits_{1\le j<m\le k}(v_m^2-v_j^2)(u_m^2-u_j^2) \right)\left( \prod\limits_{j=1}^{k}v_ju_j \right)}{\left(\prod_{j=1}^{k}(2j-1)!\right)}
\left( 1+\frac{1}{n}+O\left( n^{-5/3} \right) \right)
\]
as $n\to\infty$.
\end{corollary}
A weaker version of Corollary~\ref{cor:lockstep_fixed} (leading term + error bound) for the special case $u_j=2a_j+1$, $j=1,\dots,k$, implicitly appears in Rubey~\cite[Proof of Theorem~4.1, Chapter~2]{rubey}.
Other special instances of Corollary~\ref{cor:lockstep_fixed} can be found in \cite[Theorem~15]{MR1801472} (again in a weaker form).
For example, let us consider the so-called \emph{$k$-watermelon configuration}. In this case, the walkers start at
$(0,0),(0,2),\dots,(0,2k-2)$ and, after $2n$ steps, end at $(2n,0),(2n,2),\dots,(2n,2k-2)$.
Hence, setting $u_j=v_j=2j-1$, $1\le j\le k$, as well as replacing $n$ with $2n$ in the asymptotics above, we obtain the following stronger version of Krattenthaler et al.~\cite[Theorem~15]{MR1801472}.
\begin{corollary}
The number of $k$-watermelon configurations of length $2n$ is asymptotically equal to
\[
4^{kn}2^{k^2-k}\pi^{-k/2}n^{-k^2-k/2}\left( \prod_{j=1}^{k}(2j-1)! \right)\left( 1+\frac{1}{n}+O\left( n^{-5/2} \right) \right),
\qquad n\to\infty.
\]
\end{corollary}

Asymptotics for the number of walkers with a free end point can be derived from Theorem~\ref{thm:asymptotic_u->}.
\begin{corollary}\label{cor:lockstep_free}
The number of vicious walkers of length $n$ that start at $(0,u_1-1),\dots,(0,u_k-1)$, $0<u_1<\dots<u_k$, $u_j¸\equiv u_\ell \mod 2$,  is asymptotically equal to
\[
2^{nk+k/2}\pi^{-k/2}n^{-k^2/2}
\left( \prod_{j=1}^k\frac{u_j(j-1)!}{(2j-1)!} \right)\left( \prod_{1\le j<m\le k}(u_m^2-u_j^2) \right)
\left( 1+O\left( n^{-1} \right) \right)
, \qquad n\to\infty.
\]
\end{corollary}
Setting $u_j=2a_j+1$, $j=1,\dots,k$ in the corollary above, we obtain as a special case a stronger version of \cite[Theorem 4.1, Chapter 2]{rubey}.

The set of \emph{$k$-star configurations} consists of all possible vicious walks with the starting points $(0,0),(0,2),\dots,(0,2k-2)$.
Hence, setting $u_j=2j-1$, $j=1,\dots,k$, in the corollary above, we obtain Krattenthaler et al.~\cite[Theorem 8]{MR1801472}, but with an improved error bound.
\begin{corollary}
	The number of $k$-star configurations of length $n$ is asymptotically equal to  
\[
2^{nk+k^2-k/2}\pi^{-k/2}n^{-k^2/2}
\prod_{j=1}^k(j-1)!
\left( 1+O\left( n^{-1} \right) \right)
, \qquad n\to\infty.
\]
\end{corollary}
\subsection{Random turns model of vicious walkers with wall restriction}
\label{sec:applications:randomturns_model}
This model is quite similar to the lock step model of vicious walkers. The difference here is, that
at each time step exactly one walker is allowed to move (all the other walkers have to stay in place).

We consider the random turns model with $k$ vicious walkers. Again, at no time any two of the walkers may share a lattice point, and none of them
is allowed to go below the horizontal axis. Now, fix two points $\vec u,\vec v\in\Z^k\cap\mathcal W^0$, and assume that
for $1\le j\le k$, the $j$-th walker starts at $(0,u_j-1)$ and, after $n$ steps, ends at $(n,v_j-1)$.
In an analogous manner as in the previous subsection, we interpret this as a lattice walk of length $n$ in $\Z^k\cap\mathcal W^0$
that starts at $\vec u$ and ends at $\vec v$.
Here, the underlying lattice is given by $\mathcal L=\Z^k$ and the atomic step set is seen to be
\[ \mathcal S=\left\{ \pm \vec e^{(1)},\pm\vec e^{(2)},\dots,\pm\vec e^{(k)} \right\}. \]
The composite step set is, as in the last subsection, the set of all sequences of length one of elements in $\mathcal A$.
Since both sets, $\mathcal S$ and $\mathcal A$, satisfy Assumption~\ref{ass:step_sets}, we may obtain asymptotics by means of Theorem~\ref{thm:asymptotic_u->v} and Theorem~\ref{thm:asymptotic_u->}.

From the description of $\mathcal S$ above it is seen that the associated composite step generating function is given by
\[ S(z_1,\dots,z_k)=A(z_1,\dots,z_k)=\sum_{j=1}^{k}\left( z_j+\frac{1}{z_j} \right). \]
Short calculations give us $S(1,\dots,1)=2k$ and $S''(1,\dots,1)=2$.
Furthermore, it is easily checked that the set of maximal points is given by $\mathcal M=\left\{ (0,\dots,0),(\pi,\dots,\pi) \right\}$, which
implies $|\mathcal M|=2$.
Consequently, according to Theorem~\ref{thm:asymptotic_u->v}, we have the following result.
\begin{corollary}\label{cor:randomturns_fixed}
The number of $k$ vicious walkers in the random turns
model, where the $j$-th walker starts at $(0,u_j-1)$ and, after $n$ steps ends at $(n,v_j-1)$, is asymptotically equal to
\[
2(2k)^n\left( \frac{2}{\pi} \right)^{k/2}\left( \frac{k}{n} \right)^{k^2+k/2}
\frac{\left( \prod\limits_{1\le j<m\le k}(v_m^2-v_j^2)(u_m^2-u_j^2) \right)\left( \prod\limits_{j=1}^{k}v_ju_j \right)}{\left(\prod_{j=1}^{k}(2j-1)!\right)}
\left( 1+\frac{k}{n}+O\left( n^{-5/3} \right) \right)
\]
as $n\to\infty$.
\end{corollary}

Asymptotics for the number of vicious walks starting in $(0,u_j-1)$, $j=1,\dots,k$, with a free end point can be determined with the help of Theorem~\ref{thm:asymptotic_u->}.
\begin{corollary}\label{cor:randomturns_free}
The number of $k$-vicious walkers in the random turns model, where the $j$-th walker starts at $(0,u_j-1)$, of length $n$ is asymptotically equal to
\[
(2k)^n\left(\frac{2}{\pi}\right)^{k/2}\left( \frac{k}{n} \right)^{k^2}
\left( \prod_{j=1}^k\frac{u_j(j-1)!}{(2j-1)!} \right)\left( \prod_{1\le j<m\le k}(u_m^2-u_j^2) \right)
,\qquad n\to\infty.
\]
\end{corollary}
%
%
%

\subsection{$k$-non-crossing tangled diagrams with isolated points}
\label{sec:applications:tangleddiags:isolated}
Tangled diagrams are certain special embeddings of graphs over the vertex set $\left\{ 1,2,\dots,n \right\}$ and
vertex degrees of at most two. More precisely, the vertices are arranged in increasing order on a horizontal line,
and all edges are drawn above this horizontal line with a particular notion of crossings and nestings.
Instead of giving an in-depth presentation of tangled diagrams we refer to the papers \cite{MR2426149,ChQiReiZeil} for details,
and quote the following crucial observation by Chen et al.~\cite[Observation 2, page 3]{ChQiReiZeil}:
\begin{quote}
	``The number of $k$-non-crossing tangled diagrams over $\left\{ 1,2,\dots,n \right\}$ (allowing isolated points), equals
	the number of simple lattice walks in $x_1\ge x_2\ge\dots\ge x_{k-1}\ge 0$, from the origin back to the origin,
	taking $n$ days, where at each day the walker can either feel lazy and stay in place, or make one unit step in any (legal)
	direction, or else feel energetic and make any two consecutive steps (chosen randomly).''
\end{quote}

In order to simplify the presentation, we replace $k$ with $k+1$, and determine asymptotics for the number of $(k+1)$-non-crossing tangled diagrams.
A simple change of the lattice path description given above shows the applicability of Theorem~\ref{thm:asymptotic_u->v} to this problem. 
We proceed with a precise description.
Consider a typical walk of the type described in the quotation above, and let $\big( (c_1^{(m)},\dots,c_k^{(m)}) \big)_{m=0,\dots,n}$
be the sequence of lattice points visited during the walk.
Then, the sequence $\big( (c_k^{(m)}+1,c_{k-1}^{(m)}+2,\dots,c_1^{(m)}+k) \big)_{m=0,\dots,n}$ is sequence of lattice points visited by a walker starting and ending in $(1,2,\dots,k)$ that is confined to the region $0<x_1<x_2<\dots<x_k$ with the same step set as described in the quotation above.
This clearly defines a bijection between walks of the type described in the quotation above and walks confined to the region $0<x_1<\dots<x_k$ starting and ending in $\vec u=(1,2,\dots,k)$ with the same set of steps.

As a consequence, we see that the number of $(k+1)$-non-crossing tangled diagrams with isolated points on the set $\left\{ 1,2,\dots,n \right\}$ is equal to the number of walks starting and ending in $\vec u$ that are confined to the region $0<x_1<\dots<x_k$ and consist of composite steps from the set 
\[ \mathcal S=\left\{ \vec 0 \right\}\cup\mathcal A\cup\mathcal A\times\mathcal A, \]
where the atomic step set $\mathcal A$ is given by
\[ \mathcal A=\left\{ \pm \vec e^{(1)},\pm\vec e^{(2)},\dots,\pm\vec e^{(k)} \right\}. \]
The step sets $\mathcal A$ and $\mathcal S$ are seen to satisfy the assumptions of Theorem~\ref{thm:asymptotic_u->v}, and, therefore, may be used to obtain asymptotics for $P_n^+(\vec u\to\vec u)$.

According to the definition of the composite step set $\mathcal S$, the composite step generating function $S(z_1,\dots,z_k)$ is given by
\[
S(z_1,\dots,z_k) =
1+\left( \sum_{j=1}^{k}z_j+\frac{1}{z_j} \right)+\left( \sum_{j=1}^{k}z_j+\frac{1}{z_j} \right)^2.
\]
Short calculations show that $S(1,\dots,1)=1+2k+4k^2$ and $S''(1,\dots,1)=2+8k$, and it is easily seen that $(1,\dots,1)$ is
the only point of maximal modulus of $S(z_1,\dots,z_k)$ on the torus $|z_1|=\dots=|z_k|=1$.
Consequently, Theorem~\ref{thm:asymptotic_u->v} gives us asymptotics for the number of $(k+1)$-non-crossing tangled diagrams.
\begin{corollary}
The total number of $(k+1)$-non-crossing tangled diagrams
is asymptotically equal to
\[
(1+2k+4k^2)^n\left( \frac{2}{\pi} \right)^{k/2}\left( \frac{1+2k+4k^2}{n(2+8k)} \right)^{k^2+k/2}\left( \prod_{j=1}^{k}(2j-1)! \right)
\left( 1+\frac{1+2k+4k^2}{2n(1+4k)}+O\left( n^{-5/3} \right) \right)
\]
as $n\to\infty$.
\end{corollary}
\subsection{$k$-non-crossing tangled diagrams without isolated points}
\label{sec:applications:tangleddiags}
Consider a tangled diagram as defined in the previous example.
A vertex of this tangled diagram is called \emph{isolated}, if and only if its vertex degree is zero, that is, the vertex is isolated in the graph theoretical sense.

Again, for the sake of convenience, we shift $k$ by one, and consider $(k+1)$-non-crossing tangled diagrams without isolated
points.
In an analogous manner as in the previous section, these diagrams can be bijectively mapped onto a set of lattice paths (see \cite[Observation 1, p.3]{ChQiReiZeil}) in the region $0<x_1<\dots<x_k$ that start and end in $\vec u=(1,2,\dots,k)$.
The only difference to the situation described in the last example is the fact, that now the walker is not allowed to 
stay in place.
Hence, the composite step set $\mathcal{S}$ is now given by
\[ \mathcal S=\mathcal A\cup\mathcal A\times\mathcal A. \] 
The atomic step set $\mathcal A$ remains unchanged.

According to the definition of $\mathcal S$, the composite step generating function is now given by
\[
S(z_1,\dots,z_k) =
\left( \sum_{j=1}^{k}z_j+\frac{1}{z_j} \right)+\left( \sum_{j=1}^{k}z_j+\frac{1}{z_j} \right)^2,
\]
so that $S(1,\dots,1)=2k+4k^2$ and $S''(1,\dots,1)=2+8k$, as well as $\mathcal M=\left\{ (0,\dots,0) \right\}$.
Asymptotics for the number of $(k+1)$-non-crossing tangled diagrams without isolated points can now easily be determined with the help of Theorem~\ref{thm:asymptotic_u->v}.
\begin{corollary}
The total number of $(k+1)$-non-crossing tangled
diagrams without isolated points is asymptotically equal to
\[
(2k+4k^2)^n\left( \frac{2}{\pi} \right)^{k/2}\left( \frac{2k+4k^2}{n(2+8k)} \right)^{k^2+k/2}\left( \prod_{j=1}^{k}(2j-1)! \right)
\left( 1+\frac{1+2k^2}{n(1+4k)}+O\left( n^{-5/3} \right) \right)
\]
as $n\to\infty$.
\end{corollary}

\appendix

\section{Selberg type integrals}
\label{app:selberg}

In this section, we collect some useful results concerning integrals of the form
\[
\selberg{f(\vec x)}=\int\limits_\Gamma\!\cdots\!\int\limits_\Gamma f(\vec x)\Phi(\vec x)d\vec x,
\]
where $\vec x = (x_1,\dots,x_k)$, $d\vec x=d x_1\cdots d x_k$ and either
\[
\Phi(\vec x)=\Phi_L(\vec x)=\left( \prod_{j=1}^k\sqrt{x_j} \right)\left( \prod_{1\le j<m\le k}(x_m-x_j) \right)^2e^{-\sum_{j=1}^kx_j}
\qquad\textrm{and}\qquad
\Gamma=[0,\infty)
\]
or
\[
\Phi(\vec x)=\Phi_H(\vec x)=\left( \prod_{1\le j<m\le k}(x_m-x_j) \right)^2e^{-\sum_{j=1}^kx_j^2/2}
\qquad\textrm{and}\qquad
\Gamma=(-\infty,\infty).
\]
For $f(x)=1$, both integrals are special cases of the well-known Selberg integral (see, e.g., \cite[Chapter 17]{MR2129906}) with respect to the Laguerre weight and the Hermite weight, respectively (this should explain the subscripts).
The corresponding integrals will be denoted by $\selberg{f(\vec x)}_L$ and $\selberg{f(\vec x)}_H$, respectively.

\begin{proposition}[see {\cite[Section 17.6]{MR2129906}}]
	We have
	\begin{equation}
		\selberg{1}_L = \pi^{k/2}2^{-k^2}\prod_{j=1}^k(2j-1)!
		\label{eq:selberg1_L}
	\end{equation}
	and
	\begin{equation}
		\selberg{1}_H = (2\pi)^{k/2}\prod_{j=1}^kj!
		\label{eq:selberg1_H}
	\end{equation}
	\label{prop:selberg_eval}
\end{proposition}

\begin{proposition}
	We have the evaluations
	\begin{equation}
		\selberg{\sum_{j=1}^kx_j}_L = \left( k^2+\frac{k}{2} \right)\selberg{1}_L.
		\label{eq:selberg2_L}
	\end{equation}
	and
	\begin{equation}
		\selberg{\sum_{j=1}^kx_j^2}_H = k^2\selberg{1}_H.
		\label{eq:selberg2_H}
	\end{equation}
	\label{prop:selberg_sum_squarded_average}
\end{proposition}
\begin{proof}
	The idea underlying this proof is due to Aomoto (for details and references see, e.g., \cite[Section 17.3]{MR2129906}).
	We calculate
	\[
	\frac{\partial}{\partial x_m}x_m\Phi_L(\vec x) = \Phi_L(\vec x)\left( \frac{3}{2}+2\sum_{j\neq m}\frac{x_m}{x_m-x_j}-x_m \right).
	\]
	Now, integrating gives us
	\[
	0 = \selberg{\frac{3}{2}}_L+2\sum_{j\neq m}\selberg{\frac{x_m}{x_m-x_j}}_L-\selberg{x_m}_L.
	\]
	Finally, summing over $m$ gives us
	\[
	0 = \selberg{\frac{3k}{2}}_L+\selberg{2\binom{k}{2}}_L-\selberg{\sum_{m=1}^kx_m}_L,
	\]
	which proves Equation~\eqref{eq:selberg2_L}.

	In order to prove Equation~\eqref{eq:selberg2_H}, we proceed analogously by calculating
	\[
\frac{\partial}{\partial x_m}x_m\Phi_H(\vec x) = \Phi_H(\vec x)\left( 1+2\sum_{j\neq m}\frac{x_m}{x_m-x_j}-x_m^2 \right).
	\]
	Again, integrating and summing over $m$ yields
	\[
	0 = \selberg{k}_H+\selberg{2\binom{k}{2}}_H-\selberg{\sum_{m=1}^kx_m^2}_H,
	\]
	which proves the claim.
\end{proof}
\bibliographystyle{plain}
\bibliography{randomwalks}

\end{document}